\documentclass[11pt,reqno]{amsart}
\usepackage[a4paper,top=3cm, bottom=3cm,left=3cm, right=3cm, twoside]{geometry}

\usepackage{amsthm}
 \usepackage{amsmath}
 \usepackage{amsfonts}
 \usepackage{amssymb}
 \usepackage{amscd}
 \usepackage[T1]{fontenc}
 \usepackage[utf8x]{inputenc}
 \usepackage[english]{babel}
 \usepackage[mathscr]{eucal}
 \usepackage{lmodern}
 \usepackage{bera}
 \usepackage{float} 
 \usepackage{bold-extra}
 \usepackage{textcomp}
 \usepackage{graphicx}
 \usepackage{caption}
\usepackage{subcaption} 
 \usepackage[dvipsnames,table]{xcolor}
 \usepackage{mathtools} 
 \usepackage{enumitem}
 \usepackage{tikz} 
 \usepackage{tikz-cd}
 \usepackage{wasysym}
 \usepackage{todonotes}
  \usepackage{fancybox}
 \usepackage{fancyhdr}
\usepackage{hyperref}

 \theoremstyle{definition}  
  \newtheorem{definition}{Definition}[section]
  \newtheorem{example}[definition]{Example}

   \newtheorem{note}[definition]{Note}
   \newtheorem{remark}[definition]{Remark}

  \theoremstyle{plain}  
  \newtheorem{theorem}[definition]{Theorem}
  \newtheorem{lemma}[definition]{Lemma}
  \newtheorem{proposition}[definition]{Proposition}
  \newtheorem{corollary}[definition]{Corollary}

 \renewcommand{\it}[1]{\textit{#1}}
 
 \renewcommand{\sf}[1]{\textsf{#1}}

 \newcommand{\mbb}[1]{\mathbb{#1}}
 \newcommand{\mcl}[1]{\mathcal{#1}}

 \newcommand{\msc}[1]{\mathscr{#1}}

 \newcommand{\ol}[1]{\overline{#1}}
 \newcommand{\ul}[1]{\underline{#1}}
 \newcommand{\wtilde}[1]{\widetilde{#1}}

 \newcommand{\norm}[1]{\left\lVert#1\right\rVert}

 \newcommand{\M}[1]{\mbb{M}_{#1}}

 \newcommand{\B}[1]{\msc{B}({#1})}

 \newcommand{\ip}[1]{\langle#1\rangle}

 \newcommand{\ran}[1]{\sf{range}(#1)}
 \newcommand{\clran}[1]{\ol{\sf{range}}(#1)}
 \renewcommand{\ker}[1]{\sf{ker}(#1)}

 \newcommand{\mscriptsize}[1]{{\setlength{\arraycolsep}{.3ex}\text{\scriptsize$#1$}}}

 \newcommand{\Matrix}[1]{\begin{bmatrix}#1\end{bmatrix}}
 \newcommand{\sMatrix}[1]{\mscriptsize{\Matrix{#1}}}

 \DeclareMathOperator{\T}{\sf{T}}

 \DeclareMathOperator{\id}{\sf{id}}

 \DeclareMathOperator{\lspan}{\sf{span}}
 \DeclareMathOperator{\cspan}{\ol{\lspan}}

 \numberwithin{equation}{section}
 \allowdisplaybreaks[4] 
 \setlist[enumerate]{font=\upshape,noitemsep, topsep=0pt} 
 \setlist[itemize]{noitemsep, topsep=0pt}

\makeatletter
\@namedef{subjclassname@2020}{\textup{2020} Mathematics Subject Classification}
\makeatother

\begin{document}

\title[$C^*$-extreme points]{$C^*$-extreme contractive completely positive maps}

\author{Anand O. R}
\address{Indian Institute of Technology Madras, Department of Mathematics, Chennai, Tamilnadu 600036, India}
\email{anandpunartham@gmail.com, ma20d006@smail.iitm.ac.in}

\author{K. Sumesh}
\address{Indian Institute of Technology Madras, Department of Mathematics, Chennai, Tamilnadu 600036, India}
\email{sumeshkpl@gmail.com, sumeshkpl@iitm.ac.in}

\date{\today}

\begin{abstract} 
  In this paper we generalize a specific quantized convexity structure of the generalized state space of a $C^*$-algebra and examine the associated extreme points. We introduce the notion of $P$-$C^*$-convex subsets, where $P$ is any positive operator on a Hilbert space $\mathcal{H}$. These subsets are defined with in the set of all completely positive (CP) maps from a unital $C^*$-algebra $\mathcal{A}$ into the algebra $\mathscr{B}(\mathcal{H})$ of bounded linear maps on $\mathcal{H}$. In particular, we focus on  certain $P$-$C^*$-convex sets, denoted by $\mathrm{CP}^{(P)}(\mathcal{A},\mathscr{B}(\mathcal{H}))$, and analyze their extreme points with respect to this new convexity structure. This generalizes the existing notions of $C^*$-convex subsets and $C^*$-extreme points of unital completely positive maps. We significantly extend many of the known results regarding the $C^*$-extreme points of  unital completely positive maps into the context of $P$-$C^*$-convex sets we are considering. This includes abstract characterization and structure of $P$-$C^*$-extreme points. Further, we discuss the connection between $P$-$C^*$-extreme points and linear extreme points of these convex sets, as well as Krein-Milman type theorems. Additionally, using these studies, we completely characterize the $C^*$-extreme points of the $C^*$-convex set of all contractive completely positive maps from $\mathcal{A}$ into $\mathscr{B}(\mathcal{H})$, where $\mathcal{H}$ is finite-dimensional.  
\end{abstract}



\maketitle


\section{Introduction}

 The classical notion of convexity has been generalized to the non-commutative (or quantum) framework over time. Some of these quantum notions include  $C^*$-convexity (\cite{LoPa81,FaMo97}), matrix convexity (\cite{Ost95,EfWi97}), nc-convexity (\cite{DaKe22}) and CP-convexity (\cite{Fuj93}). We focus on $C^*$-convexity,  where the idea is to substitute scalar coefficients in a convex combination with $C^*$-algebra valued coefficients. 
 
 Loebl and Paulsen (\cite{LoPa81}) first defined $C^*$-convexity structure and the $C^*$-extreme points for subsets of $C^*$-algebras. Later, Farenick and Morenz (\cite{FaMo97}) introduced and studied the $C^*$-convexity structure and the $C^*$-extreme points of the convex set $\mathrm{UCP}(\mcl{A},\B{\mcl{H}})$ of all unital completely positive (UCP) maps from a unital $C^*$-algebra $\mcl{A}$ into the $C^*$-algebra $\B{\mcl{H}}$ of all bounded linear operators on a Hilbert space $\mcl{H}$. (In literature, the set $\mathrm{UCP}(\mcl{A},\B{\mcl{H}})$  is referred to as the generalized state space of $\mcl{A}$, and is also denoted as $S_\mcl{H}(\mcl{A})$.) In \cite{FaMo97}, the authors completely described the $C^*$-extreme points of $\mathrm{UCP}(\mcl{A},\B{\mcl{H}})$ when $\mcl{A}$ is either a commutative $C^*$-algebra or a finite dimensional matrix algebra, and $\mcl{H}$ is a finite-dimensional Hilbert space. Subsequently, Farenick and Zhou (\cite{FaZh98}) extended  this work and provided the structure of the $C^*$-extreme points of $\mathrm{UCP}(\mcl{A},\B{\mcl{H}})$, where $\mcl{A}$ is any unital $C^*$-algebra and  $\mcl{H}$ is any finite-dimensional Hilbert space. They demonstrated that in these cases, the $C^*$-extreme points can be expressed as direct sums of pure UCP maps that satisfy certain "nested" properties. In \cite{FaZh98, Zho98}, abstract characterization of $C^*$-extreme points of $\mathrm{UCP}(\mcl{A},\B{\mcl{H}})$ is provided in the general case. 

 When $\mcl{A}$ is a commutative unital $C^*$-algebra and $\mcl{H}$ is any arbitrary Hilbert space, Gregg (\cite{Gre09}) presented necessary conditions for $C^*$-extreme points of $\mathrm{UCP}(\mcl{A},\B{\mcl{H}})$ in relation to positive operator valued measures. Following this approach, Banerjee et. al. (\cite{BBK21}) exploited the connection between UCP maps on commutative unital $C^*$-algebra and positive operator valued measures, and in particular studied $C^*$-extreme points of UCP maps on commutative unital $C^*$-algebra. Recently, Bhat and Kumar (\cite{BhKu22}) extensively studied the structure of $C^*$-extreme points of $\mathrm{UCP}(\mcl{A},\B{\mcl{H}})$ in a general context. They significantly extended a result of \cite{FaZh98} from finite-dimensional to infinite-dimensional Hilbert space set up. Furthermore, they established a connection between $C^*$-extreme points and the factorization property of an associated nest algebra, provided examples of $C^*$-extreme UCP maps and discussed their applications. 

 Across all the studies mentioned above, the fundamental aim has been to understand the structure of the $C^* $-extreme points within the $C^*$-convex set being analyzed, and to find an analogue of the Krein-Milman theorem for $C^*$-convexity under an appropriate topology. The quantum analogue of the Krein-Milman theorem is known to hold for the $C^*$-convex set $\mathrm{UCP}(\mcl{A},\B{\mcl{H}})$ in the following scenarios:
 (i) when $\mcl{A}$ is an arbitrary unital $C^*$-algebra and $\mcl{H}$ is a finite-dimensional Hilbert space (\cite{FaMo97}) 
 (ii) when $\mcl{A}$ is a commutative unital $C^*$-algebra and $\mcl{H}$ is an arbitrary Hilbert space (\cite{BBK21}), and 
 (iii) when $\mcl{A}$ is a separable unital $C^*$-algebra or a type $I$ factor and $\mcl{H}$ is a separable infinite-dimensional Hilbert space (\cite{BhKu22}). The general case has yet to be resolved.

 Several studies have examined $C^*$-convexity in various contexts; for example see \cite{BBK21,BDMS23,BaHo24} and the references listed therein.
 In this article we extend the concepts of $C^*$-convexity and $C^*$-extreme points to what we refer to as $P$-$C^*$-convexity and $P$-$C^*$-extreme points, respectively, where $P\in\B{\mcl{H}}$ is a positive operator. Our main focus is on the $P$-$C^*$-convex set defined by
 \begin{align}\label{eq-CP-P-defn}
        \mathrm{CP}^{(P)}(\mcl{A},\B{\mcl{H}}):=\{\mbox{all CP maps $\Phi:\mcl{A}\to\B{\mcl{H}}$ with $\Phi(1)=P$}\}.
 \end{align}
 Note that $\mathrm{CP}^{(P)}(\mcl{A},\B{\mcl{H}})$ is a convex set, but in general, it is not a $C^*$-convex set. Linear extreme points of $\mathrm{CP}^{(P)}(\mcl{A},\B{\mcl{H}})$ has been studied in \cite{Arv69,Cho75}.  We also consider the the $C^*$-convex set
 \begin{align*}
     \mathrm{CCP}(\mcl{A},\B{\mcl{H}}): =\{\mbox{all contractive CP maps from $\mcl{A}$ into $\B{\mcl{H}}$}\} 
                                        =\bigcup_P\mathrm{CP^{(P)}}(\mcl{A},\B{\mcl{H}}),
 \end{align*}
 where the union is taken over all positive contractions $P\in\B{\mcl{H}}$. Now, by analyzing the $P$-$C^*$-extreme points of $\mathrm{CP}^{(P)}(\mcl{A},\B{\mcl{H}})$  we characterize $C^*$-extreme points of contractive completely positive (CCP) maps in the case where $\mcl{H}$ is finite-dimensional Hilbert space.  
    
    This paper is organized as follows. We start Section \ref{sec-intro} by fixing some notations, defining basic terminologies and discussing some basic results that we require later. In Section \ref{sec-PC-convexity} we introduce the notion of $P$-$C^*$-convex sets and $P$-$C^*$-extreme points of subsets of $\mathrm{CP}(\mcl{A},\B{\mcl{H}})$, the set of all completely positive (CP) maps from $\mcl{A}$ into $\B{\mcl{H}}$. In particular, we concentrate on the sets $\mathrm{CP^{(P)}}(\mcl{A},\B{\mcl{H}})$. The first main theorem of this section is the abstract characterization (Theorem \ref{thm-RNT-CP-P}) of $P$-$C^*$-extreme point of $\mathrm{CP^{(P)}}(\mcl{A},\B{\mcl{H}})$. This theorem is a significant generalization of \cite[theorem 3.1.5]{Zho98}. We prove (Proposition \ref{prop-PC-ext-non-empty}) that $\mathrm{CP^{(P)}}(\mcl{A},\B{\mcl{H}})$ has sufficiently enough linear and $P$-$C^*$-extreme points. We are particularly interested in the case where $\mcl{H}$ is finite-dimensional; in such cases $P$-$C^*$-extreme points of $\mathrm{CP^{(P)}}(\mcl{A},\B{\mcl{H}})$ are linear extreme points as well (Proposition \ref{prop-Cstar-linear-ext-finite}). To understand the structure of $P$-$C^*$-extreme points of $\mathrm{CP^{(P)}}(\mcl{A},\B{\mcl{H}})$ for finite-dimensional $\mcl{H}$, we note that, as stated in Proposition \ref{prop-Phi-Phi0}, it suffices to consider the case where $P\in\B{\mcl{H}}$ is invertible. Furthermore, if $P$ is invertible operator on any Hilbert space $\mcl{H}$, then in Theorem \ref{thm-P-Cext-char}, we establish that $P$-$C^*$-extreme points of $\mathrm{CP}^{(P)}(\mcl{A},\B{\mcl{H}})$ are precisely the invertible conjugates of $C^*$-extreme points of $\mathrm{UCP}(\mcl{A},\B{\mcl{H}})$. Along with the result of Farenick and Zhou, this will completely characterize the structure of $P$-$C^*$-extreme points of $\mathrm{CP}^{(P)}(\mcl{A},\B{\mcl{H}})$ where $\mcl{H}$ is of finite-dimension (see Theorem \ref{thm-P-Cstar-structure}). We also prove a Krein-Milman type theorem for $P$-$C^*$-convexity of $\mathrm{CP^{(P)}}(\mcl{A},\B{\mcl{H}})$. See Theorem \ref{thm-KM-CP-P} and Corollary \ref{cor-KM-CP-P-fd}.  Using the results obtained, we investigate the structure of $C^*$-extreme points of $\mathrm{CCP}(\mcl{A},\B{\mcl{H}})$ in Section \ref{sec-CCP-maps}. Theorem \ref{thm-CCP-C-star-ext} of this section asserts that if $\mcl{H}$ is finite-dimensional, then the $C^*$-extreme points of $\mathrm{CCP}(\mcl{A},\B{\mcl{H}})$ are the union of all $P$-$C^*$-extreme points of $\mathrm{CP^{(P)}}(\mcl{A},\B{\mcl{H}})$, where the union is taken over all projections $P\in\B{\mcl{H}}$. This theorem provides insight into the structure of $C^*$-extreme points of CCP maps (Remark \ref{rmk-CCP-Cstar-structure}).  As an application of this theorem, we prove a Krein-Milman type theorem (Theorem \ref{thm-KMT-CCP}) for $C^*$-convexity of $\mathrm{CCP}(\mcl{A},\B{\mcl{H}})$. Finally, we examine the relation between $C^*$-extreme points and linear extreme points, and in particular prove (Corollary \ref{rmk-CCP-Cstar-linear-ext}) that $C^*$-extreme points of $\mathrm{CCP}(\mcl{A},\B{\mcl{H}})$ are linear extreme when $\mcl{H}$ is finite-dimensional.

\section{Preliminaries and basic results}\label{sec-intro}
    Unless otherwise stated, in this article we assume that $\mcl{A}$ is a unital $C^*$-algebra, $\mcl{H}$ is a complex Hilbert space and $\B{\mcl{H}}$ denote the $C^*$-algebra of all bounded linear operators on $\mcl{H}$. We let $\mcl{A}_+:=\{a^*a: a\in\mcl{A}\}$, the set of \it{positive elements} of $\mcl{A}$; similarly we define $\B{\mcl{H}}_+$. If $\mcl{A}=\B{\mcl{H}}$, then $T\in\B{\mcl{H}}_+$ if and only if $\ip{x,Tx}\geq 0$ for all $x\in\mcl{H}$. (We follow the physicists convention that inner-product is linear in the second variable and anti-linear in the first variable.) A linear map $\Phi:\mcl{A}\to\B{\mcl{H}}$ is said to be positive if $\Phi(\mcl{A}_+)\subseteq\B{\mcl{H}}_+$, and is said to be unital if $\Phi(1)=I$, where $1=1_\mcl{A}\in\mcl{A}$ is the multiplicative identity of $\mcl{A}$, and $I=I_\mcl{H}$ is the identity map on $\mcl{H}$. We say $\Phi$ is a \it{completely positive} (CP) map if
        $\sum_{i,j=1}^nT_i^*\Phi(a_i^*a_j)T_j\in\B{\mcl{H}}_+$
    for any finite subsets $\{a_j\}_{j=1}^n\subseteq\mcl{A}$ and $\{T_j\}_{j=1}^n\subseteq\B{\mcl{H}}$; equivalently, the map $\id_k\otimes\Phi:\M{k}\otimes \mcl{A}\to\M{k}\otimes\B{\mcl{H}}$ is a positive map for all $k\in\mbb{N}$, where $\id_k$ is the identity map on the matrix algebra $\M{k}$ of all $k\times k$ complex matrices. We let $\mathrm{CP}(\mcl{A},\B{\mcl{H}})$ denotes the set of all completely positive maps from $\mcl{A}$ into $\B{\mcl{H}}$, and $\mathrm{UCP}(\mcl{A},\B{\mcl{H}})$ denotes the set of all unital completely positive (UCP) maps from $\mcl{A}$ into $\B{\mcl{H}}$. For the basic theory of $C^*$-algebras we refer to \cite{Mur90}, and for the basic theory of completely positive maps, we refer to \cite{Arv69, Pau02, Bha07}.
    
    The celebrated Stinespring dilation theorem (\cite{Sti55}) states that a linear map $\Phi:\mcl{A}\to\B{\mcl{H}}$ is CP if and only if there exists a triple $(\mcl{K},\pi,V)$, called \it{Stinespring dilation}, consisting of a complex Hilbert space $\mcl{K}$, a representation (i.e., unital $\ast$-homomorphism) $\pi:\mcl{A}\to\B{\mcl{K}}$ and a bounded linear operator $V:\mcl{H}\to\mcl{K}$ satisfying $\Phi(a)=V^*\pi(a)V$ for all $a\in\mcl{A}$. Note that $\Phi$ is unital if and only if $V$ is an isometry. A dilation $(\mcl{K},\pi,V)$ is said to be \it{minimal} if $\mcl{K}=\cspan\{\pi(\mcl{A})V(\mcl{H})\}$, which is unique up to unitary equivalence. If $\mcl{H}=\mbb{C}$, the operator $V$ can be identified with the element $z:=V(1)\in\mcl{K}$, so that the Stinespring dilation reduces to the \it{GNS representation} of the positive linear functional $\Phi:\mcl{A}\to\mbb{C}$. 

    Suppose $\Phi,\Psi\in\mathrm{CP}(\mcl{A},\B{\mcl{H}})$. If $\Phi-\Psi$ is also a CP map, then we say $\Psi$ is \it{dominated by} $\Phi$ and write $\Psi\leq_{cp}\Phi$. The following result can be thought of as a Radon-Nikodym type theorem in the context of CP maps:

\begin{theorem}[{\cite[Theorem 1.4.2]{Arv69}}]\label{thm-Arv-RNT}
    Let $\Phi,\Psi\in\mathrm{CP}(\mcl{A},\B{\mcl{H}})$ and $(\mcl{K},\pi,V)$ be the minimal Stinespring dilation of $\Phi$. Then $\Psi\leq_{cp}\Phi$ if and only if there exists a unique positive contraction $D\in\pi(\mcl{A})'$, the commutant of $\pi(\mcl{A})$ inside $\B{\mcl{K}}$, such that $\Psi(a)=V^*D\pi(a)V$ for all $a\in\mcl{A}$. 
\end{theorem}

 The above theorem holds even when the Stinespring dilation $(\mcl{K},\pi,V)$ is not minimal. However, in such cases, the operator $D\in\pi(\mcl{A})'\subseteq\B{\mcl{K}}$ need not be unique. A map $\Phi\in\mathrm{CP}(\mcl{A},\B{\mcl{H}})$ is said to be \it{pure} if the only CP maps dominated by $\Phi$ are of the form $t\Phi$, where $t\in [0,1]\subseteq\mbb{R}$. Suppose $(\mcl{K},\pi,V)$ is the minimal Stinespring dilation of $\Phi$. Then, by \cite[Corollary 1.4.3]{Arv69}, $\Phi$ is pure if and only if $\pi:\mcl{A}\to\B{\mcl{K}}$ is irreducible (i.e., $\pi(\mcl{A})'=\mbb{C}I$). 
 Given an operator $T\in\B{\mcl{H}}$, consider the map $\mathrm{Ad}_T:\B{\mcl{H}}\to\B{\mcl{H}}$ defined by 
    \begin{align*}
        \mathrm{Ad}_T(X):=T^*XT,\qquad\forall~X\in\B{\mcl{H}}.
    \end{align*}
    Then $\mathrm{Ad}_T$ is a pure CP map.

    Now, let $\msc{C}\subseteq\mathrm{CP}(\mcl{A},\B{\mcl{H}})$ be any non-empty subset. By a \it{$C^*$-convex combination} of  $\Phi_j\in\msc{C},1\leq j\leq n$, we always mean a sum of the form 
       $\sum_{j=1}^n\mathrm{Ad}_{T_j}\circ\Phi_j$,
    where $T_j\in\B{\mcl{H}}$ are such that $\sum_{j=1}^nT_j^*T_j=I$.  Further, if $T_j$'s are invertible  such a sum is called a \it{proper} $C^*$-convex combination. We say $\msc{C}$ is a \it{$C^*$-convex set} if it is closed under $C^*$-convex combinations. Note that 
    $\mathrm{UCP}(\mcl{A},\B{\mcl{H}})$ and $\mathrm{CCP}(\mcl{A},\B{\mcl{H}})$ are non-empty $C^*$-convex sets that are compact in the bounded-weak (BW) topology (see \cite{Pau02, Arv69}). 
    Recall that a CP-map $\Phi:\mcl{A}\to\B{\mcl{H}}$ is contractive (i.e., $\norm{\Phi}\leq 1$) if and only if $\Phi(1)\leq I$.

    Note that a $C^*$-convex set $\msc{C}\subseteq \mathrm{CP}(\mcl{A},\B{\mcl{H}})$ is necessarily a convex set (i.e., $\sum_{j=1}^nt_j\Phi_j\in\msc{C}$, whenever $\Phi_j\in\msc{C}$ and $t_j\in[0,1]\subseteq\mbb{R}, 1\leq j\leq n$ with $\sum_{j=1}^nt_j=1$.). If $\msc{C}\subseteq \mathrm{CP}(\mcl{A},\B{\mcl{H}})$ is a non-empty convex set, then $\Phi\in\msc{C}$ is said to be a \it{(linear) extreme point} of $\msc{C}$ if whenever $\Phi$ is written as a proper convex combination, say $\Phi=\sum_{j=1}^nt_j\Phi_j$ with $\Phi_j\in\msc{C}$ and $t_j\in(0,1)$, then  $\Phi_j=\Phi$ for all $j=1,2,\cdots,n$. 

\begin{definition}\label{defn-C-extr}
   Let $\msc{C}\subseteq \mathrm{CP}(\mcl{A},\B{\mcl{H}})$ be a non-empty $C^*$-convex set. An element $\Phi\in\msc{C}$ is said to be a \it{$C^*$-extreme point} of $\msc{C}$ if whenever $\Phi$ is written as a proper $C^*$-convex combination, say 
        $$\Phi=\sum_{j=1}^n\mathrm{Ad}_{T_j}\circ\Phi_j$$
    with $\Phi_j\in\msc{C}$ and $T_j\in\B{\mcl{H}}$, then each $\Phi_j$ is unitarily equivalent to $\Phi$, i.e.,  there exist unitaries $U_j\in\B{\mcl{H}}$ such that $\Phi_j=\mathrm{Ad}_{U_j}\circ\Phi$ for all $j=1,2,\cdots,n$.
\end{definition}

  Note that any map unitarily equivalent to a $C^*$-extreme point of $\msc{C}$ is also a $C^*$-extreme point of $\msc{C}$. We denote the set of all $C^*$-extreme (respectively, linear-extreme) points of  a $C^*$-convex (respectively, convex) set $\msc{C}$ by $\msc{C}_{C^*-ext}$ (respectively, $\msc{C}_{ext}$). The following result was first observed for the $C^*$-convex set $\mathrm{UCP}(\mcl{A},\B{\mcl{H}})$ by Zhou in \cite{Zho98}. The same proof applies here as well. 
 
\begin{proposition}
    Let $\msc{C}\subseteq\mathrm{CP}(\mcl{A},\B{\mcl{H}})$ be a $C^*$-convex set and $\Phi\in\msc{C}$. Then  $\Phi \in  \msc{C}_{C^*-ext}$ if and only if, whenever $\Phi$ decomposes as
    \begin{align*}
        \Phi=\sum_{j=1}^2\mathrm{Ad}_{T_j}\circ\Phi_j,
    \end{align*}
    for some $\Phi_j\in\msc{C}$ and $T_j\in\B{\mcl{H}}$ invertible with $\sum_{j=1}^2T_j^*T_j=I$, then there exist unitaries $U_j\in\B{\mcl{H}}$ such that $\Phi_j=\mathrm{Ad}_{U_j}\circ\Phi$ for $j=1,2$.
\end{proposition}

   
 It is known that $\mathrm{UCP}_{C^*-ext}(\mcl{A,\B{\mcl{H}}})$ is always non-empty; in fact pure UCP maps and unital $\ast$-homomorphisms are both $C^*$-extreme points and linear extreme points (\cite{FaMo97, Sto63}). 
 
\begin{proposition}[\cite{FaMo97}]\label{prop-FaMo}
    If $\dim(\mcl{H})<\infty$, then the following hold:
    \begin{enumerate}[label=(\roman*)]
        \item $\mathrm{UCP}_{C^*-ext}(\mcl{A,\B{\mcl{H}}})\subseteq\mathrm{UCP}_{ext}(\mcl{A,\B{\mcl{H}}})$; 
        \item If $\Phi\in\mathrm{UCP}_{C^*-ext}(\mcl{A,\B{\mcl{H}}})$, then $\Phi$ is unitarily equivalent to a  direct sum of pure CP-maps; 
        \item If $\mcl{A}$ is a commutative  unital  $C^*$-algebra, then $\Phi\in\mathrm{UCP}_{C^*-ext}(\mcl{A,\B{\mcl{H}}})$ if and only if $\Phi$ is a unital $\ast$-homomorphism. 
    \end{enumerate}
\end{proposition}

It is unknown whether statement (i) holds true when $\mcl{H}$ is an arbitrary Hilbert space. In \cite[page 1470]{FaZh98}, an example of a linear extreme point of $\mathrm{UCP}(\M{3},\B{\mbb{C}^4})$ is provided that is not a $C^*$-extreme point. In general, direct sum of pure UCP maps need not be a $C^*$-extreme point of $\mathrm{UCP}(\mcl{A,\B{\mcl{H}}})$ (see \cite[Example 1]{FaMo97}). However, in \cite{FaMo97}, sufficient conditions are provided under which the direct sum of CP maps is a $C^*$-extreme point.  
This direction of investigation was further explored in \cite{BhKu22}. In \cite[Theorem 7.7]{BBK21}, Banerjee et. al. proved that if $\mcl{A}$ is a commutative unital $C^*$-algebra with countable spectrum and $\mcl{H}$ is a separable complex Hilbert space, then $C^*$-extreme points of $\mathrm{UCP}(\mcl{A},\B{\mcl{H}})$ are precisely unital $\ast$-homomorphisms.

    

An abstract characterization of $C^*$-extreme points of $\mathrm{UCP}(\mcl{A,\B{\mcl{H}}})$ is provided in \cite[Theorem 3.1.5]{Zho98}. However, \cite[Corollary 2.5]{BhKu22} points out a minor error in the statement, and provides the following corrected version:
 
 \begin{theorem}\label{thm-Zhou} 
    Given $\Phi\in\mathrm{UCP}(\mcl{A,\B{\mcl{H}}})$ the following are equivalent: 
    \begin{enumerate}[label=(\roman*)]
        \item $\Phi\in\mathrm{UCP}_{C^*-ext}(\mcl{A,\B{\mcl{H}}})$;
        \item For any $\Psi\in\mathrm{CP}(\mcl{A},\B{\mcl{H}})$ with  $\Psi\leq_{cp}\Phi$ and $\Psi(1)$ invertible, there exists invertible operator $Z\in\B{\mcl{H}}$ such that $\Psi=\mathrm{Ad}_Z\circ\Phi$.
    \end{enumerate}
\end{theorem}

 Let $\Phi\in\mathrm{UCP}(\mcl{A},\B{\mcl{H}})$. A unital CP map $\Psi:\mcl{A}\to\B{\mcl{G}}$, where $\mcl{G}$ is a Hilbert space, is said to be a \it{compression} of $\Phi$ if there exists an isometry $W: \mcl{G} \to \mcl{H}$ such that $\Psi= \mathrm{Ad}_W\circ\Phi$. A \it{nested sequence of compressions} of a representation $\pi:\mcl{A}\to\B{\mcl{H}}$  is a sequence $\{\Phi_j\}_j$ of UCP maps $\Phi_j:\mcl{A}\to\B{\mcl{H}_j}$ such that $\Phi_{j+1}$ is a compression of $\Phi_j$, for each $j\geq 1$, and $\Phi_1$ is a compression of $\pi$.

\begin{theorem}[{\cite[Theorem 2.1]{FaZh98}}]\label{thm-MAIN-FaZh}
    Let $\mcl{H}$ be a finite-dimensional Hilbert space, $\mcl{A}$ be a unital $C^*$-algebra and $\Phi \in \mathrm{UCP}(\mcl{A},\B{\mcl{H}})$. Then $\Phi \in \mathrm{UCP}_{C^*-ext}(\mcl{A},\B{\mcl{H}})$ if and only if there exist finitely many pairwise non-equivalent irreducible representations $\pi_1,\pi_2,\dots,\pi_k$ of $\mcl{A}$ and nested sequences of compressions $\Phi_j^{\pi_i}$ $(1\leq j \leq n_i)$ of each representation $\pi_i$ such that $\Phi$ is unitarily equivalent to the direct sum 
    \begin{align*}
        \bigoplus_{i=1}^k \Big(\bigoplus_{j=1}^{n_i} \Phi_j^{\pi_i}\Big)
    \end{align*}
    of pure UCP maps $\Phi_j^{\pi_i}:\mcl{A}\to\B{\mcl{H}_j^i}$. (Here the above direct sum is with respect to the decomposition $\mcl{H}=\oplus_{i=1}^k\oplus_{j=1}^{n_i}\mcl{H}_j^i$.)
\end{theorem}

The above theorem was first proved in \cite[Theorem 4.1]{FaMo97} in the special case where $\mcl{A}=\M{n}$. Example 2 of \cite{FaMo97} illustrates that the conditions of the above theorem are no longer necessary for an element of $\mathrm{UCP}_{C^*-ext}(\mcl{A,\B{\mcl{H}}})$ when $\mcl{H}$ is of infinite-dimension. 


\section{\texorpdfstring{$P$-$C^*$}{P-C*}-extreme points}\label{sec-PC-convexity} 

In this section we introduce the concepts of $P$-$C^*$-convex sets and $P$-$C^*$-extreme points, where $P\in\B{\mcl{H}}$ is a positive operator. These definitions remain valid even if we replace $\B{\mcl{H}}$ with any arbitrary $C^*$-algebra.  

\begin{definition}
    Let $P\in\B{\mcl{H}}_+$ and $T_j\in\B{\mcl{H}},1\leq j\leq n$ be such that  $\sum_{j=1}^nT_j^*PT_j=P$. If $\Phi_j\in\mathrm{CP}(\mcl{A},\B{\mcl{H}}),1\leq j\leq n$, then a  sum of the form
    \begin{align*}
        \sum_{j=1}^n\mathrm{Ad}_{T_j}\circ\Phi_j
    \end{align*}
    is called a \it{$P$-$C^*$-convex combination} of $\Phi_j$'s. Such a sum is said to be \it{proper} if $T_j$'s are invertible.
\end{definition}    

      Note that if $P\in\B{\mcl{H}}_+$ is invertible and $S_j\in\B{\mcl{H}},1\leq j\leq n$, are such that $\sum_{j=1}^nS_j^*S_j=I$, then $T_j:=P^{-\frac{1}{2}}S_jP^{\frac{1}{2}}\in\B{\mcl{H}}$ satisfies the identity $\sum_{j=1}^nT_j^*PT_j=P$. Now, if $P\in\B{\mcl{H}}_+$ is arbitrary, the following lemma implies that that given any scalar $t\in(0,1)$ and invertible $X\in\B{\mcl{H}}$ with $X^*PX\leq P$, there always exists an invertible $Y\in\B{\mcl{H}}$ such that $P=t X^*PX+Y^*PY$.
    
 \begin{lemma}\label{lem-Douglas-cor}
    Let $P,Q\in\B{\mcl{H}}$ be such that $0\leq Q \leq P$. Then for any scalar $t\in(0,1)$ there exists invertible $Y\in \B{\mcl{H}}$ such that $P- t Q = Y^*PY$.
\end{lemma} 

\begin{proof}
    Since $0\leq (1-t)P \leq P-t Q \leq P$, from Douglas' Lemma (\cite{Dou66}), it follows that $\ran{(P-t Q)^{\frac{1}{2}}}=\ran{P^{\frac{1}{2}}}$. Hence, by Theorem \ref{thm-Fill-Dix}, there exists invertible $Y\in \B{\mcl{H}}$ such that $(P-t Q)^{\frac{1}{2}}= P^{\frac{1}{2}} Y$. Then $Y^*PY = (P^{\frac{1}{2}} Y)^*(P^{\frac{1}{2}} Y)= P-t Q$.
\end{proof}   
    
\begin{definition}
    Let $P\in\B{\mcl{H}}_+$. A non-empty set  $\msc{C}\subseteq\mathrm{CP}(\mcl{A},\B{\mcl{H}})$ that  is closed under any $P$-$C^*$-convex combination is called a \it{$P$-$C^*$-convex set}.
\end{definition}    
    
 Every $P$-$C^*$-convex set is a convex set. (For, let $\msc{C}$ be a $P$-$C^*$-convex set, $\Phi_j\in\msc{C}, t_j\in [0,1],1\leq j\leq n$ be such that $\sum_jt_j=1$. Then letting $T_j=\sqrt{t_j}I$ we have $\sum_jT_j^*PT_j=P$ and hence $\sum_jt_j\Phi_j=\sum_j\mathrm{Ad}_{T_j}\circ\Phi_j\in\msc{C}$.)  If $P=I$, then  a $P$-$C^*$-convex combination and a $P$-$C^*$-convex set reduce to a $C^*$-convex combination and a $C^*$-convex set, respectively.  
    
    Let $P\in\B{\mcl{H}}_+$. We observe that the set $\mathrm{CP}^{(P)}(\mcl{A},\B{\mcl{H}})$, defined as in \eqref{eq-CP-P-defn}, is a non-empty $P$-$C^*$-convex set  that is compact in the BW topology (\cite[Theorem 7.4]{Pau02}).  Clearly,  $\mathrm{CP}^{(I)}(\mcl{A},\B{\mcl{H}})=\mathrm{UCP}(\mcl{A},\B{\mcl{H}})$. 

\begin{definition}\label{defn-P-C-star-ext} 
    Let $P\in\B{\mcl{H}}_+$ and $\Phi\in\mathrm{CP}^{(P)}(\mcl{A},\B{\mcl{H}})$. Then $\Phi$ is said to be a \it{$P$-$C^*$-extreme point} of the $P$-$C^*$-convex set $\mathrm{CP}^{(P)}(\mcl{A},\B{\mcl{H}})$ if whenever
    \begin{align*}
        \Phi=\sum_{j=1}^n\mathrm{Ad}_{T_j}\circ\Phi_j
    \end{align*}
    is a proper $P$-$C^*$-convex combination of $\Phi$ with $\Phi_j\in\mathrm{CP}^{(P)}(\mcl{A},\B{\mcl{H}})$ and $T_j\in\B{\mcl{H}}$, then there exist invertible elements $S_j\in\B{\mcl{H}}$ such that $\Phi_j=\mathrm{Ad}_{S_j}\circ\Phi$ for all $1\leq j\leq n$.   
\end{definition}

Note that since invertible isometries are unitary, the definition of $C^*$-extreme points of UCP maps in Definition \ref{defn-P-C-star-ext}  coincides with the definition provided in \cite{FaMo97}. We let $\mathrm{CP}^{(P)}_{C^*-ext}(\mcl{A},\B{\mcl{H}})$ and $\mathrm{CP}^{(P)}_{ext}(\mcl{A},\B{\mcl{H}})$ denote the set of all $P$-$C^*$-extreme points and linear extreme points, respectively, of $\mathrm{CP}^{(P)}(\mcl{A},\B{\mcl{H}})$. 

 Let $T\in\B{\mcl{H}}$ be invertible such that $T^*PT = P$. Then any $\Phi \in \mathrm{CP}^{(P)}(\mcl{A},\B{\mcl{H}})$ can be written as a proper $P$-$C^*$-convex combination $\Phi=\mathrm{Ad}_{T_1}\circ\Phi_1+\mathrm{Ad}_{T_2}\circ\Phi_2$, where $T_j:= \frac{1}{\sqrt{2}} T^{-1}$ and $\Phi_j:=\mathrm{Ad}_{T_j}\circ\Phi\in\mathrm{CP}^{(P)}(\mcl{A},\B{\mcl{H}})$. This is one of the reason why we take invertible equivalence in the definition of $P$-$C^*$-extreme points. Furthermore, we have the following:

\begin{remark}\label{rmk-PC-ext}
    Let $\Phi \in \mathrm{CP}^{(P)}(\mcl{A},\B{\mcl{H}})$. Then, from the definition of $P$-$C^*$-extreme point, it follows that  $\Phi \in \mathrm{CP}^{(P)}_{C^*-ext}(\mcl{A},\B{\mcl{H}})$ if and only if $\mathrm{Ad}_T\circ\Phi \in \mathrm{CP}^{(P)}_{C^*-ext}(\mcl{A},\B{\mcl{H}})$ for all $T\in\B{\mcl{H}}$ invertible with $T^*PT = P$. 
\end{remark}

\begin{remark}
    Let $P\in\B{\mcl{H}}_+$. Then $\mathrm{CP}^{(P)}(\mcl{A},\B{\mcl{H}})$ is a $C^*$-convex set if and only if $P\in\mbb{C}I$. For, assume that $\mathrm{CP}^{(P)}(\mcl{A},\B{\mcl{H}})$ is a $C^*$-convex set. Choose and fix a map $\Phi\in\mathrm{CP}^{(P)}(\mcl{A},\B{\mcl{H}})$. Then for all unitary $U\in\B{\mcl{H}}$ we must have $\mathrm{Ad}_U\circ\Phi\in\mathrm{CP}^{(P)}(\mcl{A},\B{\mcl{H}})$ so that $PU=UP$. Since unitary elements span the set $\B{\mcl{H}}$ it follows that $P=\lambda I$ for some $\lambda\in\mbb{C}$. Conversely, if $P=\lambda I$ for some $\lambda\in\mbb{C}$ and let $\Phi_j\in\mathrm{CP}^{(P)}(\mcl{A},\B{\mcl{H}}), T_j\in\B{\mcl{H}}$ with $\sum_{j=1}^nT_j^*T_j=I$. Then $\sum_{j=1}^nT_j^*PT_j=P$ so that $\sum_{j=1}^n\mathrm{Ad}_{T_j}\circ\Phi_j\in\mathrm{CP}^{(P)}(\mcl{A},\B{\mcl{H}})$ concluding that $\mathrm{CP}^{(P)}(\mcl{A},\B{\mcl{H}})$ is a $C^*$-convex set.
\end{remark} 

 Suppose $P=\lambda I$ for some $\lambda\in\mbb{C}$. Then, by the definition, the $P$-$C^*$-extreme points and the $C^*$-extreme points of $\mathrm{CP}^{(P)}(\mcl{A},\B{\mcl{H}})$ coincide.  In fact, as we will see later in Lemma \ref{lem-proj-unitary}, when $P$ is a projection, the definition of $P$-$C^*$-extreme points can be reformulated in terms of unitary equivalence. This means that, in such a case, we can choose the $S_j$'s in Definition \ref{defn-P-C-star-ext}   to be unitary operators.

\begin{theorem}[{\cite[Corollary 1]{FiWi71},\cite[Theorem 2.2]{Dix49}}]\label{thm-Fill-Dix} 
    Let $A,B\in\B{\mcl{H}}$. Then there exists an invertible operator $C\in\B{\mcl{H}}$ such that $A=BC$ if and only $\ran{A}=\ran{B}$ and $\ker{A}=\ker{B}$. In particular, two positive operators differ by an invertible operator if and only if they have the same range. 
\end{theorem}

\begin{lemma}\label{lem-invertible-conjugate-positive-inequality}
    Let $T \in \B{\mcl{H}}$ be invertible and $P \in \B{\mcl{H}}_+$ be such that $T^*PT \leq\beta P$ for some scalar $\beta>0$. Then the following are equivalent:
    \begin{enumerate}[label=(\roman*)]
        \item $\alpha P \leq T^*PT$ for some scalar $\alpha >0$;
        \item $\ran{T^*P^{\frac{1}{2}}} = \ran{P^{\frac{1}{2}}}$; 
        \item $T^*P^{\frac{1}{2}} = P^{\frac{1}{2}} Y$ for some invertible $Y \in \B{\mcl{H}}$.
    \end{enumerate}
\end{lemma}

\begin{proof}
    $(i) \Rightarrow (ii)$ Since $\alpha P \leq T^*PT \leq \beta P$, by Douglas' Lemma,  $\ran{T^*P^{\frac{1}{2}}} = \ran{P^{\frac{1}{2}}}$. \\ 
     $(ii) \Rightarrow (iii)$
     Assume that $\ran{T^*P^{\frac{1}{2}}} = \ran{P^{\frac{1}{2}}}$. Since $T^*$ is invertible, we have $\ker{T^*P^{\frac{1}{2}}}=\ker{P^{\frac{1}{2}}}$. Then, by Theorem \ref{thm-Fill-Dix}, there exists an invertible $Y \in \B{\mcl{H}}$ such that $T^*P^{\frac{1}{2}} = P^{\frac{1}{2}} Y$.\\
      $(iii) \Rightarrow (i)$
     Suppose $Y \in \B{\mcl{H}}$ invertible is such that $T^*P^{\frac{1}{2}} = P^{\frac{1}{2}} Y$. Since $YY^*$ is positive and invertible,  there exists a scalar $\alpha>0$ such that $YY^* \geq \alpha I$. 
      Then 
    \begin{align*}
        T^*PT = (T^*P^{\frac{1}{2}})(T^*P^{\frac{1}{2}})^* 
              = (P^{\frac{1}{2}}Y)(P^{\frac{1}{2}}Y)^* 
              = P^{\frac{1}{2}}YY^*P^{\frac{1}{2}} 
            \geq P^{\frac{1}{2}}(\alpha I)P^{\frac{1}{2}} 
            = \alpha P.
    \end{align*} 
   This completes the proof.  
\end{proof}

\begin{remark}\label{rem-invertible-conjugate-positive-inequality}
    Let $T,P\in\B{\mcl{H}}$ be as in the above lemma. If $\dim(\mcl{H})<\infty$ or $P$ is invertible, then the assertions of the above lemma hold.  For, if $P$ is invertible then clearly (ii) holds. Now, suppose $\dim(\mcl{H})<\infty$. Since $T^*PT\leq\beta P$, by Douglas' lemma, we have $T^*(\ran{P^{\frac{1}{2}}}) =\ran{T^*P^{\frac{1}{2}}}\subseteq\ran{P^{\frac{1}{2}}}$. Since $T^*$ is injective and $\ran{P^{\frac{1}{2}}}$ is finite-dimensional we must have $T^*(\ran{P^{\frac{1}{2}}})=\ran{P^{\frac{1}{2}}}$.    
\end{remark}

\begin{example}
    Let $\mcl{H} = l^2(\mbb{Z})$ and $P,Q$ be the projections onto $\cspan\{e_n\}_{n\geq0}$ and $\cspan\{e_n\}_{n\geq1}$, respectively. Let $T\in\B{\mcl{H}}$ be the bilateral left shift operator which is invertible. Then $Q=T^*PT \leq P$. But there does not exist any $\beta > 0$ such that $\beta P \leq Q$.
\end{example}

\begin{proposition}
    Let $P\in\B{\mcl{H}}_+$. Suppose either $P$ is invertible or $\dim(\mcl{H})<\infty$. Then the following are equivalent:
    \begin{enumerate}[label=(\roman*)]
        \item $\Phi\in \mathrm{CP}^{(P)}_{C^*-ext}(\mcl{A},\B{\mcl{H}})$;
        \item For any proper $P$-$C^*$-convex combination of $\Phi$, say $\Phi=\sum_{j=1}^2\mathrm{Ad}_{T_j}\circ\Phi_j$, with $\Phi_j\in\mathrm{CP}^{(P)}(\mcl{A},\B{\mcl{H}})$ and $T_j\in\B{\mcl{H}}$,  there exist invertible elements $S_j\in\B{\mcl{H}}$ such that $\Phi_j=\mathrm{Ad}_{S_j}\circ\Phi$ for $j=1,2$.
    \end{enumerate}
\end{proposition} 

\begin{proof}
	To prove the nontrivial implication $(ii) \Rightarrow (i)$, consider a  proper $P$-$C^*$-convex decomposition of $\Phi$, say $\Phi=\sum_{j=1}^n\mathrm{Ad}_{T_j}\circ\Phi_j$, with $T_j\in\B{\mcl{H}}$ and $\Phi_j\in\mathrm{CP}^{(P)}(\mcl{A},\B{\mcl{H}})$. To show that $\Phi_j=\mathrm{Ad}_{S_j}\circ\Phi$ for some invertible $S_j\in\B{\mcl{H}}, 1\leq j\leq n$. We prove by induction on $n\in\mbb{N}$. By assumption, the result is true for $n=2$. Assume that the result is true for $n-1$.  Now to prove the result for  $n$. If $P$ is invertible or $\dim(\mcl{H}) < \infty$, then since $T_j^*PT_j\leq P$, by Remark \ref{rem-invertible-conjugate-positive-inequality}, there exists $\alpha_j>0$ such that $\alpha_j P\leq T_j^*PT_j$ for each $2\leq j\leq n$. Letting $\alpha=\sum_{j=2}^n\alpha_j>0$ we get $\alpha P \leq \sum_{j=2}^n T_j^*PT_j \leq P$. Hence, by Douglas' lemma, $\ran{(\sum_{j=2}^n T_j^*PT_j)^{\frac{1}{2}}} = \ran{P^{\frac{1}{2}}}$. So, by Theorem \ref{thm-Fill-Dix}, there exists an invertible $X \in \B{\mcl{H}}$ such that $(\sum_{j=2}^n T_j^*PT_j)^{\frac{1}{2}} = P^{\frac{1}{2}} X$. Then
    \begin{align*}
        T_1^*PT_1 + X^*PX = T_1^*PT_1 + (P^{\frac{1}{2}}X)^*(P^{\frac{1}{2}}X)
                          =T_1^*PT_1+\sum_{j=2}^n T_j^*PT_j
                         = P.
    \end{align*}
    Now, consider $\Psi := \sum_{j=2}^n\mathrm{Ad}_{T_jX^{-1}}\circ\Phi_j\in \mathrm{CP}^{(P)}(\mcl{A},\B{\mcl{H}})$. Note that $\Phi = \mathrm{Ad}_{T_1}\circ\Phi_1 + \mathrm{Ad}_{X}\circ\Psi$. So by hypothesis $\Phi_1 = \mathrm{Ad}_{S_1}\circ\Phi$ and $\Psi=\mathrm{Ad}_S\circ\Phi$ for some invertible $S,S_1 \in \B{\mcl{H}}$. Then
    \begin{align*}
        \Phi=\mathrm{Ad}_{S^{-1}}\circ\Psi=\sum_{j=2}^n\mathrm{Ad}_{T_jX^{-1}S^{-1}}\circ\Phi_j
    \end{align*}is a proper $P$-$C^*$-convex combination of $\Phi_j$'s and hence, by induction hypothesis, there exist invertible $S_j\in\B{\mcl{H}}$ such that $\Phi_j=\mathrm{Ad}_{S_j}\circ\Phi$ for $2\leq j\leq n$. This completes the proof.  
\end{proof}


 In the following theorem, we generalize Theorem \ref{thm-Zhou} to the context of $\mathrm{CP}^{(P)}(\mcl{A,\B{\mcl{H}}})$ using Zhou's techniques.
 
\begin{theorem}\label{thm-RNT-CP-P} 
    Let $P\in\B{\mcl{H}}_+$ and $\Phi\in\mathrm{CP}^{(P)}(\mcl{A,\B{\mcl{H}}})$. Then the following are equivalent:
    \begin{enumerate}[label=(\roman*)]
        \item $\Phi\in\mathrm{CP}^{(P)}_{C^*-ext}(\mcl{A,\B{\mcl{H}}})$;
        \item For any $\Psi\in\mathrm{CP}(\mcl{A},\B{\mcl{H}})$ with  $\Psi\leq_{cp}\Phi$ and $\Psi(1)=B^*PB$ for some invertible $B\in\B{\mcl{H}}$, there exists invertible operator $Z\in\B{\mcl{H}}$ such that $\Psi=\mathrm{Ad}_Z\circ\Phi$.
    \end{enumerate}
\end{theorem}

\begin{proof}
    $(i)\Rightarrow (ii)$ Suppose $\Phi\in \mathrm{CP}^{(P)}_{C^*-ext}(\mcl{A},\B{\mcl{H}})$ and $(\mcl{K},\pi,V)$ is the minimal Stinespring representation of $\Phi$. Let $\Psi\in\mathrm{CP}(\mcl{A},\B{\mcl{H}})$ be such that $\Psi \leq_{cp} \Phi$ and $\Psi(1)=B^*PB$ for some invertible $B\in\B{\mcl{H}}$. Then, by  Theorem \ref{thm-Arv-RNT}, there exists a positive contraction $D\in\pi(\mcl{A})'\subseteq\B{\mcl{K}}$ such that $\Psi(\cdot)=V^*D\pi(\cdot)V$. Now, let $t\in(0,1)$. Then by Lemma \ref{lem-Douglas-cor}, there exists an invertible operator $C\in\B{\mcl{H}}$ such that  $P-tB^* PB= C^*PC$. Set $S_1=t D$ and $S_2=I-t D$. Then $S_1,S_2\in\pi(\mcl{A})'$ are positive contractions such that $S_1+S_2 = I$, and $V^*S_1V=tB^*PB$ and $V^*S_2V=C^*PC$. Now for $j=1,2$, define $\Phi_j\in\mathrm{CP}(\mcl{A},\B{\mcl{H}})$ by
    \begin{align*}
        \Phi_j(\cdot):=(T_j^{-1})^*V^*S_j\pi(\cdot)VT_j^{-1},
    \end{align*}
    where $T_1=\sqrt{t}B$ and $T_2=C$. Then, $\Phi_j \in \mathrm{CP}^{(P)}(\mcl{A},\B{\mcl{H}})$ and 
    \begin{align*}
        \Phi(\cdot)= V^*(S_1+S_2)\pi(\cdot)V
                        =V^*S_1\pi(\cdot)V+V^*S_2\pi(\cdot)V
                        =\mathrm{Ad}_{T_1}\circ\Phi_1+\mathrm{Ad}_{T_2}\circ\Phi_2
    \end{align*}
    is a proper $P$-$C^*$-convex combination of $\Phi_1$ and $\Phi_2$. Since $\Phi\in\mathrm{CP}^{(P)}_{C^*-ext}(\mcl{A},\B{\mcl{H}})$ there must exists an invertible operator $Y\in\B{\mcl{H}}$ such that 
    \begin{align*}
        \Phi(\cdot)&=Y^*\Phi_1(\cdot)Y 
                   =Y^*({T_1}^{-1})^*V^*S_1\pi(\cdot)V{T_1}^{-1}Y\\
                   &=Y^*(B^{-1})^*V^*D\pi(\cdot)VB^{-1}Y
                   =Y^*(B^{-1})^*\Psi(\cdot)B^{-1}Y.
    \end{align*} 
    Thus $\Psi(\cdot)= \mathrm{Ad}_Z\circ\Phi(\cdot)$, where $Z:=Y^{-1}B\in\B{\mcl{H}}$ is invertible.

    $(ii)\Rightarrow (i)$ Consider a $P$-$C^*$-convex decomposition of $\Phi$, say $\Phi=\sum_{j=1}^n\mathrm{Ad}_{T_j}\circ\Phi_j$ with $\Phi_j\in\mathrm{CP}^{(P)}(\mcl{A},\B{\mcl{H}})$ and $T_j\in\B{\mcl{H}}$ invertible. Then $\Psi_j := \mathrm{Ad}_{T_j}\circ\Phi_j \leq_{cp} \Phi$ and $\Psi_j(1)=T_j^*PT_j$ with $T_j$ invertible. By hypothesis, there exist invertible operators    $Z_j \in \B{\mcl{H}}$ such that $\mathrm{Ad}_{T_j}\circ\Phi_j = \mathrm{Ad}_{Z_j}\circ\Phi$, i.e., $\Phi = \mathrm{Ad}_{T_j{Z_j}^{-1}}\circ\Phi_j$ for all $1\leq j\leq n$. This concludes that $\Phi\in\mathrm{CP}^{(P)}_{C^*-ext}(\mcl{A,\B{\mcl{H}}})$.
\end{proof}

\begin{proposition}\label{prop-PC-ext-non-empty}
    Given $P\in\B{\mcl{H}}_+$ the sets $\mathrm{CP}^{(P)}_{C^*-ext}(\mcl{A},\B{\mcl{H}})$ and $\mathrm{CP}^{(P)}_{ext}(\mcl{A},\B{\mcl{H}})$ are non-empty. In fact, the following hold: 
    \begin{enumerate}[label=(\roman*)]
        \item Let $\psi$ be a pure state on $\mcl{A}$ and $\Phi(\cdot):=\psi(\cdot)P$. Then $\Phi\in\mathrm{CP}^{(P)}_{C^*-ext}(\mcl{A},\B{\mcl{H}})$ and $\Phi\in\mathrm{CP}^{(P)}_{ext}(\mcl{A},\B{\mcl{H}})$. 
        \item If $\Phi:\mcl{A}\to\B{\mcl{H}}$ is any pure CP-map with $\Phi(1)=P$, then $\Phi\in\mathrm{CP}^{(P)}_{C^*-ext}(\mcl{A},\B{\mcl{H}})$ and $\Phi\in\mathrm{CP}^{(P)}_{ext}(\mcl{A},\B{\mcl{H}})$.
    \end{enumerate}
\end{proposition}

\begin{proof}
    $(i)$ Let $(\mcl{K},\pi,z)$ be the minimal GNS representation of $\psi$, i.e., $\mcl{K}$ is a Hilbert space, $\pi:\mcl{A}\to\B{\mcl{K}}$ is a unital $\ast$-homomorphism and $z\in\mcl{K}$ is a unit vector such that 
    \begin{align*}
        \psi(a) = \ip{z,\pi(a)z}=V^*\pi(a)V,\qquad\forall~a\in\mcl{A},
    \end{align*}
    where $V(\lambda):=\lambda z$ for all $\lambda\in\mbb{C}$. Since $\psi$ is pure, $\pi$ must be irreducible.  Define $\widetilde{\pi}:\mcl{A}\to\B{\mcl{H}\otimes\mcl{K}}$ by $\widetilde{\pi}(\cdot):=I_{\mcl{H}}\otimes\pi(\cdot)$, and define $\widetilde{V}:\mcl{H}\otimes\mbb{C}\to\mcl{H}\otimes\mcl{K}$ by $\widetilde{V}:=P^{\frac{1}{2}}\otimes V$. We identify $\mcl{H}=\mcl{H}\otimes\mbb{C}$ via $x\mapsto x\otimes 1$ to see that  $\Phi(\cdot)=\widetilde{V}^*\widetilde{\pi}(\cdot)\widetilde{V}$. Thus  $(\mcl{H}\otimes\mcl{K},\widetilde{\pi},\widetilde{V})$ is a Stinespring dilation of $\Phi$. Now, let $\Psi\in\mathrm{CP}(\mcl{A},\B{\mcl{H}})$ be such that  $\Psi\leq_{cp}\Phi$ and $\Psi(1)=B^*PB$ for some invertible $B\in\B{\mcl{H}}$.  Then, by Theorem \ref{thm-Arv-RNT}, there exists a positive contraction $\widetilde{T}\in {\widetilde{\pi}}(\mcl{A})'\subseteq\B{\mcl{H}\otimes\mcl{K}}$ such that $\Psi(\cdot)=\widetilde{V}^*\widetilde{T}\widetilde{\pi}(\cdot)\widetilde{V}$. Since $\pi(\mcl{A})' = \mbb{C} I_{\mcl{K}}$, from \cite[Theorem IV.5.9]{Tak79}, it follows that  
    \begin{align*}
        \widetilde{\pi}(\mcl{A})'
           &= (I_{\mcl{H}}\otimes_{alg}\pi(\mcl{A}))' 
            = (I_{\mcl{H}}\otimes_{alg}\ol{\pi(\mcl{A})}^{sot})' 
            = (\mbb{C}I_{\mcl{H}}\otimes_{alg}\pi(\mcl{A})'')'  \\
           &= (\mbb{C}I_{\mcl{H}}\ol{\otimes}_{alg}^{sot}\pi(\mcl{A})'')'
            =\B{\mcl{H}}\ol{\otimes}_{alg}^{sot}\pi(\mcl{A})''' 
            = \B{\mcl{H}}\ol{\otimes}_{alg}^{sot}\pi(\mcl{A})'\\
            & = \B{\mcl{H}} \otimes_{alg} I_{\mcl{K}},
    \end{align*}
    where $\otimes_{alg}$ is the algebraic tensor product. So, there exists $T\in\B{\mcl{H}}$ such that $\widetilde{T}= T\otimes I_{\mcl{K}}$. Now, $0 \leq \widetilde{T} \leq I_{\mcl{H}\otimes\mcl{K}}$ implies $0\leq T \leq I_{\mcl{H}}$. Therefore,
    \begin{align*}
        \Psi(\cdot)
            &=(P^{\frac{1}{2}}\otimes V)^*(T\otimes I_\mcl{K})(I_\mcl{H}\otimes\pi(\cdot))(P^{\frac{1}{2}}\otimes V)\\
            &= (P^{\frac{1}{2}}TP^{\frac{1}{2}}) \otimes (V^*\pi(\cdot)V) \\
            &= (P^{\frac{1}{2}}TP^{\frac{1}{2}}) \otimes \psi(\cdot) \\
            &= \psi(\cdot)(P^{\frac{1}{2}}TP^{\frac{1}{2}})\qquad(\because~\mcl{H}=\mcl{H}\otimes\mbb{C}).
    \end{align*}
    Now, $B^*PB=\Psi(1)=P^{\frac{1}{2}}TP^{\frac{1}{2}}$ implies that $\Psi(\cdot)= \psi(\cdot) B^*PB =  \mathrm{Ad}_B\circ \Phi(\cdot)$.
    Hence, by Theorem \ref{thm-RNT-CP-P}, we get $\Phi \in \mathrm{CP}^{(P)}_{C^*-ext}(\mcl{A},\B{\mcl{H}})$.

    Next, we show that $\Phi\in\mathrm{CP}^{(P)}_{ext}(\mcl{A},\B{\mcl{H}})$. So let $\Phi=\sum_{j=1}^n t_j\Phi_j$ be a proper convex decomposition of $\Phi$ with $\Phi_j \in \mathrm{CP}^{(P)}(\mcl{A},\B{\mcl{H}})$ and $t_j\in (0,1)$. We show that $\Phi_j=\Phi$ for all $1\leq j \leq n$ so that $\Phi\in\mathrm{CP}^{(P)}_{ext}(\mcl{A},\B{\mcl{H}})$. So let $x\in\mcl{H}$ and let   $\alpha_x:\B{\mcl{H}}\to\mbb{C}$ be the positive linear functional defined by $\alpha_x(\cdot)=\ip{x,(\cdot)x}$.  \\
    \ul{\sf{Case (1):}} Suppose $\ip{x,Px}=0$, i.e., $\alpha_x(P)=0$. Then $\alpha_x\circ\Phi(1)=0=\alpha_x\circ\Phi_j(1)$, so that the CP-maps $\alpha_x\circ\Phi$ and $\alpha_x\circ\Phi_j$ are identically zero maps for all $1\leq j\leq n$.\\
    \ul{\sf{Case (2):}} Suppose $\ip{x,Px} > 0$. For $1\leq j \leq n$ define the states $\psi_j^x : \mcl{A}\to\mbb{C}$ by 
    \begin{align*}
      \psi_j^x(\cdot) := \frac{\ip{x,\Phi_j(\cdot)x}}{\ip{x,Px}}.
    \end{align*}
    Then $\psi(\cdot) = \sum_{j=1}^n t_j\psi_j^x(\cdot)$ is a proper convex linear combination of the states $\psi_j^x, 1\leq j \leq n$. Since $\psi$ is pure we must have $\psi=\psi_j^x$ for all $1 \leq j \leq n$. Thus, for all $a\in\mcl{A}$,
    \begin{align*}
        \alpha_x\circ\Phi_j(a)
             =\ip{x,\Phi_j(a)x}
             =\psi_j^x(a)\ip{x,Px}
              =\ip{x,\psi(a)Px}
              =\alpha_x\circ\Phi(a).
    \end{align*}
     Thus in both the cases we have
     \begin{align*}
         \ip{x,\Phi(a)x}=\alpha_x\circ\Phi(a)=\alpha_x\circ\Phi_j(a)=\ip{x,\Phi_j(a)x},\qquad\forall~a\in\mcl{A}.
    \end{align*}
     Since $x\in\mcl{H}$ is arbitrary, from the above, we conclude that $\Phi=\Phi_j$ for all $1\leq j \leq n$.  \\
    $(ii)$ If $P=0$, the result follows trivially. So assume $P\neq 0$, and let $\Phi=\sum_{j=1}^n\mathrm{Ad}_{T_j}\circ\Phi_j$ be a proper $P$-$C^*$-convex combination of $\Phi_j \in \mathrm{CP}^{(P)}(\mcl{A},\B{\mcl{H}})$ with $T_j\in\B{\mcl{H}}$. Since $\Phi$ is pure and $\mathrm{Ad}_{T_j}\circ\Phi_j\leq_{cp}\Phi$, there exists  some scalar $s_j\in (0,1]$ such that $\mathrm{Ad}_{T_j}\circ\Phi_j=s_j\Phi$ for all $1\leq j\leq n$. (If $s_j=0$, then $P=0$.) Hence $\Phi_j=\mathrm{Ad}_{\sqrt{s_j}{T_j}^{-1} }\circ\Phi$ for $1\leq j \leq n$. Thus $\Phi \in \mathrm{CP}^{(P)}_{C^*-ext}(\mcl{A},\B{\mcl{H}})$. Now, suppose $\Phi=t\Phi_1 + (1-t)\Phi_2$ is a proper convex combination of $\Phi_j \in \mathrm{CP}^{(P)}(\mcl{A},\B{\mcl{H}})$, where $t\in (0,1)$. Then $\Phi_1\neq 0$ and $t\Phi_1 \leq_{cp} \Phi$ implies that $t\Phi_1=s\Phi$ for some scalar $s\in (0,1]$. In particular, $t P=s P$ so that $t=s$, and consequently $\Phi_1=\Phi$. Similarly we can show $\Phi_2=\Phi$. Hence, $\Phi \in \mathrm{CP}^{(P)}_{ext}(\mcl{A},\B{\mcl{H}})$.
\end{proof}

Next we prove some technical results which are very crucial for later results. 

\begin{lemma}\label{lem-P-P0-Phi-Phi0}
    Let $0\neq P\in\B{\mcl{H}}_+$ with $\ker{P}\neq \{0\}$ and $\Phi \in \mathrm{CP}^{(P)}(\mcl{A},\B{\mcl{H}})$. Then with respect to the decomposition $\mcl{H}=\mcl{H}_0\oplus\mcl{H}_0^\perp$, where $\mcl{H}_0:=\clran{P}$, the maps $P$ and $\Phi$ has block matrix form
    \begin{align}\label{eq-P-P0-Phi-Phi0}
      P=\Matrix{P_0&0\\0&0},
      \qquad\mbox{and}\qquad
      \Phi=\Matrix{\Phi_0&0\\0&0},
    \end{align}
    where $P_0\in\B{\mcl{H}_0}_+$ with $\ker{P_0}=\{0\}$ and $\Phi_0 \in \mathrm{CP}^{(P_0)}(\mcl{A},\B{\mcl{H}_0})$. Moreover, if $\ran{P}$ is closed, then $P_0$ is invertible.  (In particular, if $P$ is projection, then $P_0=I_{\mcl{H}_0}$.)
\end{lemma}

\begin{proof}
    Clearly, with respect to the decomposition $\mcl{H}= \mcl{H}_0 \oplus \mcl{H}_0^\perp$, the operator $P=\Phi(1)$ has the  block matrix form $P=\sMatrix{P_0&0\\0&0}$,
    where $P_0\in\B{\mcl{H}_0}_+$ and $\ker{P_0}=0$. Further, $\Phi$ has the form  $\Phi(\cdot)= \sMatrix{\Phi_{0}(\cdot)&\Psi(\cdot)\\ \Psi^{\ast}(\cdot)& \Phi_{1}(\cdot)}$, where $\Phi_0:\mcl{A}\to\B{\mcl{H}_0},\Phi_1:\mcl{A}\to\B{\mcl{H}_0^\perp}$ are CP-maps and $\Psi: \mcl{A}\to\B{\mcl{H}_0^\perp,\mcl{H}_0}$ and $\Psi^{\ast}:\mcl{A}\to\B{\mcl{H}_0,\mcl{H}_0^\perp}$ are bounded linear maps with $\Psi^{\ast}(a):= {\Psi(a^*)}^*$ for all $a \in \mcl{A}$.  Since $\Phi(1)=\sMatrix{P_0&0\\0&0}$ it follows that $\Phi_0\in \mathrm{CP}^{(P_0)}(\mcl{A},\B{\mcl{H}_0})$ and $\Phi_1\in\mathrm{CP}^{(0)}(\mcl{A},\B{\mcl{H}_0^\perp})=\{0\}$. Consequently, since $\Phi$ is positive, we get $\Psi(a)=0$ for all $a\in\mcl{A}_+$. Thus $\Psi=\Psi^*=0$, so that $\Phi$ has the block matrix form  
    \begin{align*}
     \Phi(a):=\Matrix{\Phi_0(a)&0\\0&0},\qquad\forall~a\in\mcl{A}. 
    \end{align*}
    Note that if $\ran{P}$ is closed, then $\ran{P_0}$ is closed so that $\ran{P_0}=\mcl{H}_0$. Thus $P_0$ is bijective and hence invertible.  This completes the proof. 
\end{proof}

\begin{lemma}\label{lem-P-S-decomp}
    Let $0\neq P\in\B{\mcl{H}}_+$ with $\ker{P}\neq \{0\}$ and let $\mcl{H}_0:=\clran{P}$. If $S\in\B{\mcl{H}}$ is an invertible operator such that $\alpha P\leq S^*PS\leq \beta P$ for some scalars $\alpha,\beta\in (0,\infty)$, then with respect to the decomposition $\mcl{H}=\mcl{H}_0\oplus\mcl{H}_0^\perp$ the operator $S$ has the block form $S=\sMatrix{S_1&0\\S_2&S_3}$ with $S_1\in\B{\mcl{H}_0}$ invertible. 
\end{lemma}

\begin{proof}
    With respect to the decomposition $\mcl{H}=\mcl{H}_0\oplus\mcl{H}_0^\perp$ we have $P=\sMatrix{P_0&0\\0&0}$ for some $P_0\in\B{\mcl{H}_0}_+$ with $\ker{P_0}=\{0\}$, and write $S=\sMatrix{S_1&S_2\\S_3&S_4}$. 
    As $\alpha P\leq S^*PS\leq \beta P$, by Lemma \ref{lem-invertible-conjugate-positive-inequality}, we get 
    \begin{align*}
        \ran{P^{\frac{1}{2}}}=\ran{S^*P^{\frac{1}{2}}}=S^*(\ran{P^{\frac{1}{2}}}). 
    \end{align*} 
    Since $S^*$ is invertible and $\mcl{H}_0=\clran{P^{\frac{1}{2}}}$, from the above, we conclude that $S^*(\mcl{H}_0)=\mcl{H}_0$.  
    Hence $S^*$ must be of the form  $S^*=\sMatrix{S_1^*&S_3^*\\0&S_4^*}$, and this would implies that $S_1^*(\mcl{H}_0)=\mcl{H}_0$. Further, $S^*$ is invertible implies that $S_1^*$ is one-one. Thus $S_1^*$ is bijective and hence invertible.  
\end{proof}

\begin{corollary}\label{cor-invertible-conjugate-positive-inequality}
    Assume $\dim(\mcl{H})<\infty$ and let $0\neq P\in\B{\mcl{H}}_+$ with $\ker{P}\neq \{0\}$ and let $\mcl{H}_0:=\clran{P}$. If $T\in\B{\mcl{H}}$ is an invertible operator such that $T^*PT\leq \beta P$ for some scalar $\beta\in (0,\infty)$, then with respect to the decomposition $\mcl{H}=\mcl{H}_0\oplus\mcl{H}_0^\perp$ the operator $T$ has the block form $T=\sMatrix{T_1&0\\T_2&T_3}$ with $T_1\in\B{\mcl{H}_0}$ invertible.
\end{corollary}

\begin{proof}
     Follows from Remark \ref{rem-invertible-conjugate-positive-inequality} and the above lemma. 
\end{proof}

The result below shows that, in the finite-dimensional case, understanding the structure of $P$-$C^*$-extreme points can be reduced to the scenario where $\ker{P}=\{0\}$, i.e.,  $P$ is invertible.
 
\begin{proposition}\label{prop-Phi-Phi0}
    Suppose $0\neq P\in\B{\mcl{H}}_+$ with $\ker{P}\neq\{0\}$ and $\Phi\in\mathrm{CP}^{(P)}(\mcl{A},\B{\mcl{H}})$. Let $P_0$ and $\Phi_0:\mcl{A}\to\B{\mcl{H}_0}$ be as in \eqref{eq-P-P0-Phi-Phi0}, where $\mcl{H}_0:=\clran{P}$. Then the following hold:
    \begin{enumerate}[label=(\roman*)]
        \item If $\Phi\in\mathrm{CP}^{(P)}_{C^*-ext}(\mcl{A},\B{\mcl{H}})$, then $\Phi_0\in\mathrm{CP}^{(P_0)}_{C^*-ext}(\mcl{A},\B{\mcl{H}_0})$.
        \item Conversely, if $\dim(\mcl{H})<\infty$ and $\Phi_0\in\mathrm{CP}^{(P_0)}_{C^*-ext}(\mcl{A},\B{\mcl{H}_0})$, then $\Phi\in\mathrm{CP}^{(P)}_{C^*-ext}(\mcl{A},\B{\mcl{H}})$.
    \end{enumerate}
    (Note that, in $(ii)$, $P_0$ is invertible.)
\end{proposition}

\begin{proof}
    $(i)$ Assume that $\Phi\in\mathrm{CP}^{(P)}_{C^*-ext}(\mcl{A},\B{\mcl{H}})$. Suppose $\Phi_0 = \sum_{j=1}^n \mathrm{Ad}_{R_j}\circ\Psi_j$ is a proper $P_0$-$C^*$-convex decomposition with $\Psi_j \in \mathrm{CP}^{(P_0)}(\mcl{A},\B{\mcl{H}_0})$ and  $R_j \in \B{\mcl{H}_0}$ invertible satisfying $\sum_{j=1}^nR_j^*P_0R_j=P_0$. Decompose $\mcl{H}=\mcl{H}_0\oplus\mcl{H}_0^\perp$. Consider $T_j:=\sMatrix{R_j&0\\0&I_{\mcl{H}_0^\perp}}\in\B{\mcl{H}}$ invertible, and $\Phi_j(\cdot):=\sMatrix{\Psi_j(\cdot)&0\\0&0} \in \mathrm{CP}^{(P)}(\mcl{A},\B{\mcl{H}})$  for all $1\leq j\leq n$. Then $\Phi = \sum_{j=1}^n \mathrm{Ad}_{T_j}\circ\Phi_j$ is a proper $P$-$C^*$-convex decomposition. Since $\Phi\in\mathrm{CP}^{(P)}_{C^*-ext}(\mcl{A},\B{\mcl{H}})$ there exists $S_j \in \B{\mcl{H}}$ invertible such that $\Phi = \mathrm{Ad}_{S_j}\circ\Phi_j$ for all $1\leq j\leq n$. In particular, $P=S_j^*PS_j$ for all $1\leq j\leq n$. Then, by Lemma \ref{lem-P-S-decomp}, each $S_j$ has the block form $S_j=\sMatrix{X_j&0\\\ast&\ast}$ with $X_j\in\B{\mcl{H}_0}$ invertible, where '$\ast$' denote unspecified entries. Now,  $\Phi = \mathrm{Ad}_{S_j}\circ\Phi_j$ implies that  $\Phi_0 = \mathrm{Ad}_{X_j}\circ\Psi_j$ for all $1\leq j\leq n$, and we conclude that $\Phi_0 \in \mathrm{CP}^{(P_0)}_{C^*-ext}(\mcl{A},\B{\mcl{H}_0})$.  \\ 
    $(ii)$ Suppose $\Phi = \sum_{j=1}^n \mathrm{Ad}_{T_j}\circ\Phi_j$ is a proper $P$-$C^*$-convex decomposition with $\Phi_j \in \mathrm{CP}^{(P)}(\mcl{A},\B{\mcl{H}})$ and $T_j\in\B{\mcl{H}}$. Note that with respect to the decomposition $\mcl{H}=\mcl{H}_0\oplus\mcl{H}_0^\perp$ each $\Phi_j$ has the block form  $\Phi_j(\cdot)=\sMatrix{\Psi_j(\cdot)&0\\0&0}$, for some $\Psi_j\in\mathrm{CP}^{(P_0)}(\mcl{A},\B{\mcl{H}_0})$. Now, since $T_j$ is invertible, ${T_j}^*PT_j \leq P$ and $\dim(\mcl{H}) < \infty$, by Corollary \ref{cor-invertible-conjugate-positive-inequality}, we have $T_j=\sMatrix{R_j&0\\\ast&\ast}$, where '$\ast$' denote unspecified entries and $R_j\in\B{\mcl{H}_0}$ invertible for all $1\leq j\leq n$. Thus, $\Phi_0 = \sum_{j=1}^n \mathrm{Ad}_{R_j}\circ\Psi_j$ is a proper $P_0$-$C^*$-convex combination of $\Psi_j \in \mathrm{CP}^{(P_0)}(\mcl{A},\B{\mcl{H}_0})$ and hence  there exist invertible operators $X_j \in \B{\mcl{H}_0}$ such that $\Phi_0 = \mathrm{Ad}_{X_j}\circ\Psi_j$ for all $1\leq j\leq n$. Then $\Phi = \mathrm{Ad}_{S_j}\circ\Phi_j$, where $S_j=\sMatrix{X_j&0\\0&I_{\mcl{H}_0^\perp}}\in\B{\mcl{H}}$ invertible. Hence, $\Phi \in \mathrm{CP}^{(P)}_{C^*-ext}(\mcl{A},\B{\mcl{H}})$.
\end{proof}

\begin{proposition}\label{prop-Phi-Phi0-ext}
     Suppose $0\neq P\in\B{\mcl{H}}_+$ with $\ker{P}\neq\{0\}$ and $\Phi\in\mathrm{CP}^{(P)}(\mcl{A},\B{\mcl{H}})$. Let $P_0$ and $\Phi_0:\mcl{A}\to\B{\mcl{H}_0}$ be as in \eqref{eq-P-P0-Phi-Phi0}, where $\mcl{H}_0:=\clran{P}$. Then, 
      $\Phi\in\mathrm{CP}^{(P)}_{ext}(\mcl{A},\B{\mcl{H}})$ if and only if $\Phi_0\in\mathrm{CP}^{(P_0)}_{ext}(\mcl{A},\B{\mcl{H}_0})$.
\end{proposition}

\begin{proof}
     This follows from the definition of linear extreme points and the fact that, with respect to the decomposition $\mcl{H}=\mcl{H}_0\oplus\mcl{H}_0^\perp$, any $\Psi\in\mathrm{CP}^{(P)}(\mcl{A},\B{\mcl{H}})$ can be written as  
     $\Psi=\sMatrix{\Psi_0&0\\0&0}$ for some  $\Psi_0\in\mathrm{CP}^{(P_0)}(\mcl{A},\B{\mcl{H}_0})$.
\end{proof}

\noindent\textbf{Notation.}
    Given any $\Phi\in\mathrm{CP}(\mcl{A},\B{\mcl{H}})$ with $\Phi(1)\in\B{\mcl{H}}$ invertible, we define $\widehat{\Phi}:\mcl{A}\to\B{\mcl{H}}$ as the map    
    \begin{align}\label{eq-Phi-hat}
         \widehat{\Phi}(a):=\Phi(1)^{\frac{-1}{2}}\Phi(a)\Phi(1)^{-\frac{1}{2}},\qquad\forall~a\in\mcl{A}.
     \end{align}
    Clearly $\widehat{\Phi}\in\mathrm{UCP}(\mcl{A},\B{\mcl{H}})$ and $\Phi=\mathrm{Ad}_T\circ\widehat{\Phi}$, where $T=\Phi(1)^{\frac{1}{2}}\in\B{\mcl{H}}$ invertible. 

 Suppose $P\in \B{\mcl{H}}_+$ is invertible and $\Phi\in\mathrm{CP}^{(P)}(\mcl{A},\B{\mcl{H}})$. We observe that $\Phi$ is a linear extreme point of $\mathrm{CP}^{(P)}(\mcl{A},\B{\mcl{H}})$ if and only if $\widehat{\Phi}$ is a linear extreme point of the convex set $\mathrm{UCP}(\mcl{A},\B{\mcl{H}})$. In the following, we show that the same conclusion holds for $P$-$C^*$-extreme points as well. 

\begin{theorem}\label{thm-P-Cext-char}
   Let $P\in\B{\mcl{H}}_+$ be invertible and $\Phi\in\mathrm{CP}^{(P)}(\mcl{A},\B{\mcl{H}})$. Then the following are equivalent: 
    \begin{enumerate}[label=(\roman*)]
        \item $\Phi\in\mathrm{CP}^{(P)}_{C^*-ext}(\mcl{A},\B{\mcl{H}})$;
        \item $\widehat{\Phi}\in\mathrm{UCP}_{C^*-ext}(\mcl{A},\B{\mcl{H}})$; 
             
        \item $\Phi=\mathrm{Ad}_{T}\circ\widetilde{\Psi}$ for some $\widetilde{\Psi}\in\mathrm{UCP}_{C^*-ext}(\mcl{A},\B{\mcl{H}})$ and $T\in\B{\mcl{H}}$ invertible with $T^*T=P$; 
        \item For any $\Psi\in\mathrm{CP}(\mcl{A},\B{\mcl{H}})$ with $\Psi\leq_{cp}\Phi$ and $\Psi(1)$ invertible, there exists invertible operator $Z\in\B{\mcl{H}}$ such that $\Psi=\mathrm{Ad}_Z\circ\Phi$.
    \end{enumerate}
\end{theorem}

\begin{proof}
    $(i)\Rightarrow (ii)$ Suppose $\Psi_j\in\mathrm{UCP}(\mcl{A},\B{\mcl{H}})$ and $T_j\in\B{\mcl{H}}$ invertible are such that $\sum_{j=1}^nT_j^*T_j=I_\mcl{H}$ and  
        $\widehat{\Phi}=\sum_{j=1}^n\mathrm{Ad}_{T_j}\circ\Psi_j$.
    Then
    \begin{align*}
        \Phi=\mathrm{Ad}_{P^{\frac{1}{2}}}\circ\widehat{\Phi}
        =\sum_{j=1}^n\mathrm{Ad}_{\widetilde{T}_j}\circ\Phi_j,
    \end{align*}
    where $\Phi_j:=\mathrm{Ad}_{P^{\frac{1}{2}}}\circ\Psi_j\in\mathrm{CP}^{(P)}(\mcl{A},\B{\mcl{H}})$ and $\widetilde{T}_j:=P^{-\frac{1}{2}}T_jP^{\frac{1}{2}}\in\B{\mcl{H}}$ invertible with $\sum_j\widetilde{T}_j^*P\widetilde{T}_j=P$. Since $\Phi$ is a $P$-$C^*$-extreme point there exist invertible $S_j\in\B{\mcl{H}}$ such that $\Phi_j=\mathrm{Ad}_{S_j}\circ\Phi$, and hence, $\Psi_j=\mathrm{Ad}_{U_j}\circ\widehat{\Phi}$, where $U_j:=P^{\frac{1}{2}}S_jP^{-\frac{1}{2}}\in\B{\mcl{H}}$ is an invertible isometry and hence a unitary for all $1\leq j\leq n$. This concludes that $\widehat{\Phi}\in\mathrm{UCP}_{C^*-ext}(\mcl{A},\B{\mcl{H}})$.  \\
    $(ii)\Rightarrow (iii)$ Follows since $\Phi=\mathrm{Ad}_{P^{\frac{1}{2}}}\circ\widehat{\Phi}$. \\
    $(iii)\Rightarrow (iv)$ Assume that $\Phi=\mathrm{Ad}_{T}\circ\widetilde{\Psi}$ for some $\widetilde{\Psi}\in\mathrm{UCP}_{C^*-ext}(\mcl{A},\B{\mcl{H}})$ and $T\in\B{\mcl{H}}$ invertible.  Let  $\Psi\in\mathrm{CP}(\mcl{A},\B{\mcl{H}})$ be such that  $\Psi(1)$ invertible and $\Psi\leq_{cp}\Phi$. Then $\mathrm{Ad}_{T^{-1}}\circ\Psi\leq_{cp}\widetilde{\Psi}$. Since $\widetilde{\Psi}\in\mathrm{UCP}_{C^*-ext}(\mcl{A,\B{\mcl{H}}})$, by Theorem \ref{thm-Zhou}, there exists invertible operator $Y\in\B{\mcl{H}}$ such that $\mathrm{Ad}_{T^{-1}}\circ\Psi=\mathrm{Ad}_Y\circ\widetilde{\Psi}$. Then $\Psi=\mathrm{Ad}_{Z}\circ\Phi$, where $Z:=T^{-1}YT\in\B{\mcl{H}}$ is invertible. \\
    $(iv)\Rightarrow (i)$ Consider a proper $P$-$C^*$-convex decomposition  of $\Phi$, say $\Phi=\sum_{j=1}^n\mathrm{Ad}_{T_j}\circ\Phi_j$ with $\Phi_j\in\mathrm{CP}^{(P)}(\mcl{A},\B{\mcl{H}})$ and $T_j\in\B{\mcl{H}}$ invertible  and satisfies $\sum_{j=1}^nT_j^*PT_j=P$. Then for each $1\leq j\leq n$, $\Psi_j:=\mathrm{Ad}_{T_j}\circ\Phi_j$ is such that $\Psi_j(1)$ is invertible and $\Psi_j\leq_{cp}\Phi$. Hence by assumption, there exist invertible operators $Z_j\in\B{\mcl{H}}$ such that $\mathrm{Ad}_{T_j}\circ\Phi_j=\mathrm{Ad}_{Z_j}\circ\Phi$ for all $1\leq j\leq n$. 
    Then $S_j:=Z_jT_j^{-1}\in\B{\mcl{H}}$ is invertible such that $\Phi_j=\mathrm{Ad}_{S_j}\circ\Phi$ for all $1\leq j\leq n$. Hence $\Phi\in\mathrm{CP}^{(P)}_{C^*-ext}(\mcl{A,\B{\mcl{H}}})$.
\end{proof}

 Note that the equivalence of $(i)$ and $(iv)$ in the above theorem will also follows from Theorem \ref{thm-RNT-CP-P}.


\begin{theorem}\label{thm-P-Cstar-structure}
    Let $\mcl{H}$ be a finite-dimensional Hilbert space, $P\in\B{\mcl{H}}_+$, $\mcl{A}$ be a unital $C^*$-algebra and $\Phi \in \mathrm{CP^{(P)}}(\mcl{A},\B{\mcl{H}})$. Then $\Phi \in \mathrm{CP^{(P)}}_{C^*-ext}(\mcl{A},\B{\mcl{H}})$ if and only if there exist
        finitely many pairwise non-equivalent irreducible representations $\pi_1,\pi_2,\dots,\pi_k$ of $\mcl{A}$,
        Hilbert spaces $\{\mcl{H}_j^i: 1\leq i\leq k, 1\leq j\leq n_i\}$, and
        pure UCP maps $\Phi_j^{\pi_i}:\mcl{A}\to\B{\mcl{H}_j^i}$
    such that $\{\Phi_j^{\pi_i}\}_{j=1}^{n_i}$ is a  nested sequence of compressions of representation $\pi_i$ for all $1\leq i\leq k$ and 
    \begin{align}\label{eq-CP-P-Cstar-structure}
        \mcl{H}=\big(\oplus_{i=1}^k\oplus_{j=1}^{n_i}\mcl{H}_j^i\big)\oplus\ran{P}^\perp,
        \qquad
           \Phi= S^*\Big(\Big(\oplus_{i=1}^k \oplus_{j=1}^{n_i} \Phi_j^{\pi_i}\Big)\oplus 0\Big)S
    \end{align}
    for some invertible operator $S\in\B{\mcl{H}}$, where $0:\mcl{A}\to\B{\ran{P}^\perp}$ is the zero map. 
\end{theorem}

\begin{proof}
    Assume $P\neq 0$. We prove the case when $P$ is not invertible. Let $\mcl{H}_0, P_0,\Phi_0$ be as in Lemma \ref{lem-P-P0-Phi-Phi0} so that $\Phi=\sMatrix{\Phi_0&0\\0&0}$. Assume that $\Phi \in \mathrm{CP^{(P)}}_{C^*-ext}(\mcl{A},\B{\mcl{H}})$. Then, from Proposition \ref{prop-Phi-Phi0} and Theorem \ref{thm-P-Cext-char}, we have  $\Phi_0=\mathrm{Ad}_X\circ\Psi$ for some $X\in\B{\mcl{H}_0}$ invertible and $\Psi\in\mathrm{UCP}_{C^*-ext}(\mcl{A},\B{\mcl{H}_0})$. Then using Theorem \ref{thm-MAIN-FaZh} we conclude that $\Phi$ has the decomposition as in \eqref{eq-CP-P-Cstar-structure}, with $S=\sMatrix{X&0\\0&I}$. 
    Conversely, assume that $\Phi$ has a decomposition  as in \eqref{eq-CP-P-Cstar-structure}. Note that $\mcl{H}_0=\ran{P}=\oplus_{i=1}^k\oplus_{j=1}^{n_i}\mcl{H}_j^i$. Let $\Psi:=\oplus_{i=1}^k \oplus_{j=1}^{n_i} \Phi_j^{\pi_i}\in\mathrm{UCP}_{C^*-ext}(\mcl{A},\B{\mcl{H}_0})$ and write $S=\sMatrix{S_1&S_2\\S_3&S_4}$ with respect to the decomposition $\mcl{H}=\mcl{H}_0\oplus\mcl{H}_0^\perp$. Now, since $P_0$ is invertible there exists positive scalar $\alpha>0$ such that $\alpha I_{\mcl{H}_0}\leq P_0\leq I_{\mcl{H}_0}$. Since $P=S^*\sMatrix{I_{\mcl{H}_0}&0\\0&0}S$ we conclude that
    \begin{align*}
        \alpha\Matrix{I_{\mcl{H}_0}&0\\0&0}\leq S^*\Matrix{I_{\mcl{H}_0}&0\\0&0}S\leq\Matrix{I_{\mcl{H}_0}&0\\0&0}.
    \end{align*}
    Then, by Lemma \ref{lem-P-S-decomp}, we get $S_1$ is invertible, and satisfies $\mathrm{Ad}_{S_1}\circ \Psi=\Phi_0$. Hence $\Phi_0$ is $C^*$-extreme point and so is $\Phi$.   
\end{proof}

\begin{proposition}\label{prop-Cstar-linear-ext-finite}
    Suppose $\dim(\mcl{H})<\infty$. Then 
    \begin{align*}
       \mathrm{CP}^{(P)}_{C^*-ext}(\mcl{A},\B{\mcl{H}})\subseteq\mathrm{CP}^{(P)}_{ext}(\mcl{A},\B{\mcl{H}})
    \end{align*}
    for all $P\in\B{\mcl{H}}_+$.
\end{proposition}

\begin{proof}
    Due to Proposition \ref{prop-Phi-Phi0} and \ref{prop-Phi-Phi0-ext}  we assume that $P\in\B{\mcl{H}}_+$ is invertible. Let $\Phi\in\mathrm{CP}^{(P)}_{C^*-ext}(\mcl{A},\B{\mcl{H}})$. Then, from Theorem \ref{thm-P-Cext-char} and \cite[Proposition 1.1]{FaMo97}, we have 
    \begin{align*}
        \widehat{\Phi}\in\mathrm{UCP}_{C^*-ext}(\mcl{A},\B{\mcl{H}})\subseteq\mathrm{UCP}_{ext}(\mcl{A},\B{\mcl{H}}).
    \end{align*}
    But we know that $\widehat{\Phi}\in\mathrm{UCP}_{ext}(\mcl{A,\B{\mcl{H}}})$ if and only if  $\Phi\in\mathrm{CP}^{(P)}_{ext}(\mcl{A},\B{\mcl{H}})$. 
\end{proof}

\begin{remark}\label{rmk-P-Cstar-comm-homo}
    Assume that $\dim(\mcl{H}) < \infty$.  Let $\mcl{A}$ be a commutative unital $C^*$-algebra, $P\in\B{\mcl{H}}$ be a projection and $\Phi\in\mathrm{CP^{(P)}}(\mcl{A},\B{\mcl{H}})$. Then $\Phi \in \mathrm{CP}^{(P)}_{C^*-ext}(\mcl{A},\B{\mcl{H}})$ if and only if $\Phi$ is multiplicative i.e., $\ast$-homomorphism. This follows from Propositions \ref{prop-Phi-Phi0}, \ref{prop-FaMo}, and the fact that $\Phi$ is multiplicative if and only if $\Phi_0$ in Lemma \ref{lem-P-P0-Phi-Phi0} is multiplicative. We prove a more general result in the following. 
\end{remark}

 In (\cite[Proposition 1.4.10]{Arv69}), Arveson characterized the linear extreme points of the convex set $\mathrm{CP^{(P)}}(\mcl{A},\B{\mcl{H}})$, where $\mcl{A}$ is a commutative unital $C^*$-algebra and $\dim(\mcl{H})<\infty$. In the following, we characterize the $P$-$C^*$-extreme points. 

\begin{proposition}
    Assume that $\dim(\mcl{H})<\infty$ and $\mcl{A}=C(\Omega)$, where $\Omega$ is a compact Hausdorff space. Let $P\in\B{\mcl{H}}_+$ and $\Phi\in\mathrm{CP^{(P)}}(\mcl{A},\B{\mcl{H}})$. Then the following are equivalent: 
    \begin{enumerate}[label=(\roman*)]
        \item $\Phi \in \mathrm{CP}^{(P)}_{C^*-ext}(\mcl{A},\B{\mcl{H}})$;
        \item There exist distinct points $w_1,w_2,\cdots,w_k$ in $\Omega$ and mutually orthogonal projections $\{Q_j\}_{j=1}^k\subseteq\B{\mcl{H}}$ such that 
        \begin{align}\label{eq-P-Cstar-comm}
            \Phi(f)=\sum_{j=1}^kf(w_j)P^{\frac{1}{2}}Q_jP^{\frac{1}{2}},\qquad\forall~f\in C(\Omega),
        \end{align}
        and $\sum_{j=1}^kQ_j$ is the projection onto the $\ran{P}$.
    \end{enumerate}
\end{proposition}

\begin{proof}
    If $P=0$, then nothing to prove. So assume that $P\neq0$. Because of Proposition \ref{prop-Phi-Phi0}, it is enough to prove the equivalence of $(i)$ and $(ii)$ for the case when $P$ is invertible. 
    Note that if $P$ is not invertible, then $Q_j$'s in \eqref{eq-P-Cstar-comm} satisfies the relation  $Q_j\leq \sMatrix{I_{\mcl{H}_0}&0\\0&0}$ , the projection onto $\mcl{H}_0:=\ran{P}$. Hence each $Q_j$ must have the block matrix form $Q_j=\sMatrix{Q_{j,0}&0\\0&0}$ for some $Q_{j,0}\in\B{\mcl{H}_0}$. So assume that $P$ is invertible.\\
    $(i)\Rightarrow (ii)$ Assume that $\Phi \in \mathrm{CP}^{(P)}_{C^*-ext}(\mcl{A},\B{\mcl{H}})$. Then, by Theorem \ref{thm-P-Cext-char} and \cite[Proposition 2.2]{FaMo97}, there exist distinct points $w_1,w_2,\cdots,w_k$ in $\Omega$ and mutually orthogonal projections $\{Q_j\}_{j=1}^k\subseteq\B{\mcl{H}}$ such that $\sum_jQ_j=I$ and 
        \begin{align*}
            \widehat{\Phi}(f)=\sum_{j=1}^kf(w_j)Q_j,\qquad\forall~f\in C(\Omega),
        \end{align*}
    Then \eqref{eq-P-Cstar-comm} holds as $\Phi=\mathrm{Ad}_{P^{\frac{1}{2}}}\circ\widehat{\Phi}$. \\
    $(ii)\Rightarrow (i)$ Assume that $\Phi$ has the form \eqref{eq-P-Cstar-comm}. Then 
    \begin{align*}
        \widehat{\Phi}(f)=\mathrm{Ad}_{P^{-\frac{1}{2}}}\circ\Phi(f)=\sum_{j=1}^kf(w_j)Q_j,\qquad\forall~f\in C(\Omega).
    \end{align*}
    Hence, again by \cite[Proposition 2.2]{FaMo97}, we have $\widehat{\Phi}\in\mathrm{UCP}_{C^*-ext}(\mcl{A},\B{\mcl{H}})$ so that $\Phi \in \mathrm{CP}^{(P)}_{C^*-ext}(\mcl{A},\B{\mcl{H}})$.
\end{proof}

 Next, we prove a non-commutative analogue of the classical Krein-Milman theorem in the context of the $P$-$C^*$-convex set $\mathrm{CP}^{(P)}(\mcl{A},\B{\mcl{H}})$.

\begin{definition}
    Given a subset $\mcl{S}\subseteq\mathrm{CP}(\mcl{A},\B{\mcl{H}})$ and $P\in\B{\mcl{H}}_+$ we let 
    \begin{align*}
       P\mbox{-}C^*{\mbox{-}}con(\mcl{S}):=\Big\{\sum_{j=1}^n\mathrm{Ad}_{T_j}\circ\Phi_j: \Phi_j\in\mcl{S},T_j\in\B{\mcl{H}}\mbox{ such that }\sum_{j=1}^nT_j^*PT_j=P,n\geq 1 \Big\},
    \end{align*}
    and is called the \it{$P$-$C^*$-convex hull} of $\mcl{S}$.  If $P=I$, then we denote the above set simply by $C^*\mbox{-}con(\mcl{S})$ and is called the \it{$C^*$-convex hull} of $\mcl{S}$.
 \end{definition}

 We observe that $P\mbox{-}C^*{\mbox{-}}con(\mcl{S})$ is the smallest $P\mbox{-}C^*{\mbox{-}}$convex set containing $\mcl{S}$. The following theorem generalize \cite[Theorem 3.5]{FaMo97}, and \cite[Theorem 5.3]{BhKu22} and \cite[Theorem 7.14]{BBK21}.

\begin{theorem}\label{thm-KM-CP-P}
   Let $\mcl{A}$ be a unital $C^*$-algebra and $\mcl{H}$ be a separable Hilbert space such that one of the following happens:
   \begin{enumerate}[label=(\roman*)]
       \item $\mcl{H}$ is finite-dimensional; 
       \item $\mcl{A}$ is separable or type $I$ factor; 
       \item $\mcl{A}$ is commutative. 
   \end{enumerate}
   Then for any $P\in\B{\mcl{H}}_+$ invertible,
   \begin{align*}
       \mathrm{CP}^{(P)}(\mcl{A},\B{\mcl{H}})=
       \ol{P\mbox{-}C^*{\mbox{-}}con}\Big(\mathrm{CP}^{(P)}_{C^*-ext}(\mcl{A},\B{\mcl{H}})\Big),
   \end{align*}
   where the closure is taken with respect to the BW-topology on $\mathrm{CP}(\mcl{A},\B{\mcl{H}})$. 
\end{theorem}

\begin{proof}
    Let $\Phi\in\mathrm{CP}^{(P)}(\mcl{A},\B{\mcl{H}})$. Then $\widehat{\Phi}\in\mathrm{UCP}(\mcl{A},\B{\mcl{H}})$, and hence from \cite{FaMo97,BBK21,BhKu22}, there exists a net $\{\Psi_\alpha\}_\alpha\in C^*{\mbox{-}}con(\mathrm{UCP}_{C^*-ext}(\mcl{A},\B{\mcl{H}}))$ that converges to $\widehat{\Phi}$ in BW-topology. Write $\Psi_\alpha=\sum_{j=1}^{n_\alpha}\mathrm{Ad}_{T_{j,\alpha}}\circ\Psi_{j,\alpha}$ for some $\Psi_{j,\alpha}\in\mathrm{UCP}_{C^*-ext}(\mcl{A},\B{\mcl{H}})$, $n_\alpha\in\mbb{N}$ and $T_{j,\alpha}\in\B{\mcl{H}}$ with $\sum_{j=1}^{n_\alpha} T_{j,\alpha}^*T_{j,\alpha}=I_\mcl{H}$. Then, by Theorem \ref{thm-P-Cext-char}, $\widetilde{\Psi}_{j,\alpha}:= \mathrm{Ad}_{P^{\frac{1}{2}}}\circ\Psi_{j,\alpha} \in \mathrm{CP}^{(P)}_{C^*-ext}(\mcl{A},\B{\mcl{H}})$ and hence
    \begin{align*}
        \mathrm{Ad}_{P^{\frac{1}{2}}}\circ\Psi_\alpha =\sum_{j=1}^{n_\alpha}\mathrm{Ad}_{P^{-\frac{1}{2}}T_{j,\alpha}P^{\frac{1}{2}}}\circ\widetilde{\Psi}_{j,\alpha}\in P\mbox{-}C^*{\mbox{-}}con\big(\mathrm{CP}^{(P)}_{C^*-ext}(\mcl{A},\B{\mcl{H}})\big).
    \end{align*}
    Now, $\Psi_\alpha$ converges to $\widehat{\Phi}$ in BW-topology implies that $\mathrm{Ad}_{P^{\frac{1}{2}}}\circ\Psi_\alpha $ converges to $\mathrm{Ad}_{P^{\frac{1}{2}}}\circ\widehat{\Phi}=\Phi$ in BW-topology. Thus $\Phi\in \ol{P\mbox{-}C^*{\mbox{-}}con}\Big(\mathrm{CP}^{(P)}_{C^*-ext}(\mcl{A},\B{\mcl{H}})\Big)$. 
\end{proof}

\begin{corollary}\label{cor-KM-CP-P-fd}
    If $\mcl{H}$ is finite-dimensional, then
    \begin{align*}
       \mathrm{CP}^{(P)}(\mcl{A},\B{\mcl{H}})=
       \ol{P\mbox{-}C^*{\mbox{-}}con}\Big(\mathrm{CP}^{(P)}_{C^*-ext}(\mcl{A},\B{\mcl{H}})\Big),
   \end{align*}
    for all $P\in\B{\mcl{H}}_+$,  where the closure is taken w.r.t the BW-topology on $\mathrm{CP}(\mcl{A},\B{\mcl{H}})$. 
\end{corollary}

\begin{proof}
    Follows from Proposition \ref{prop-Phi-Phi0} and Theorem \ref{thm-KM-CP-P}. 
\end{proof}

\section{\texorpdfstring{$C^*$}{C*}-extreme points of contractive CP-maps}\label{sec-CCP-maps}

 Let $\msc{C}_1,\msc{C}_2$ be two $C^*$-convex subsets of  $\mathrm{CP}(\mcl{A},\B{\mcl{H}})$ such that $\msc{C}_1\subseteq\msc{C}_2$. Then, from definition of $C^*$-extreme points, it follows that 
 \begin{align}\label{eq-C-star-ext-subset}
     (\msc{C}_2)_{C^*-ext}\bigcap\msc{C}_1\subseteq(\msc{C}_1)_{C^*-ext}.
 \end{align}

  
\begin{proposition}\label{prop-UCP-C-Cstar}
 Let $\msc{C}$ be any $C^*$-convex set such that $\mathrm{UCP}(\mcl{A},\B{\mcl{H}})\subseteq\msc{C}\subseteq\mathrm{CCP}(\mcl{A},\B{\mcl{H}})$. Then
 \begin{align*}
     \mathrm{UCP}_{C^*-ext}(\mcl{A},\B{\mcl{H}})=\mathrm{UCP}(\mcl{A},\B{\mcl{H}})\bigcap \msc{C}_{C^*-ext}.
 \end{align*}  
 \end{proposition}

 \begin{proof}
     Let $\Phi \in \mathrm{UCP}_{C^*-ext}(\mcl{A},\B{\mcl{H}})$ and $\Phi=\sum_{j=1}^2\mathrm{Ad}_{b_j}\circ\Phi_j$ for some $\Phi_j\in\msc{C}$ and $T_j\in\B{\mcl{H}}$ invertible with $\sum_{j=1}^2T_j^*T_j=I$. Then 
     \begin{align*}
         \sum_{j=1}^2T_j^*T_j = I
                              = \Phi(1)
                              = \sum_{j=1}^2T_j^*\Phi_j(1)T_j
     \end{align*}
     that is, $\sum_{j=1}^2T_j^*\big(I - \Phi_j(1)\big)T_j = 0$. Then, since each $I-\Phi_j(1)\geq 0$, we must have $T_j^*(I- \Phi_j(1))T_j = 0$. Since $T_j$'s are invertible we get  $I- \Phi_j(1)=0$, i.e., $\Phi_j\in \mathrm{UCP}(\mcl{A},\B{\mcl{H}})$. Therefore, there exist unitaries $U_j \in \B{\mcl{H}}$ such that $\Phi_j=\mathrm{Ad}_{U_j}\circ\Phi$ for  $j=1,2$, and we conclude that $\Phi \in \msc{C}_{C^*-ext}$. 
 \end{proof}

\begin{remark}\label{rmk-UCP-ext-C-ext}
    If $\msc{C}$ is a convex set as in the above proposition, in a similar way, we can show that $\mathrm{UCP}_{ext}(\mcl{A},\B{\mcl{H}}) = \mathrm{UCP}(\mcl{A},\B{\mcl{H}}) \bigcap \msc{C}_{ext}$. 
\end{remark}

Now we prove a technical lemma which we need later.

\begin{lemma}\label{lem-norm-Phi-j}
     Let $\Phi\in\mathrm{CCP}(\mcl{A},\B{\mcl{H}})$ with $\norm{\Phi}=1$. If $\Phi=\mathrm{Ad}_{T_1}\circ\Phi_1+\mathrm{Ad}_{T_2}\circ\Phi_2$ is a proper $C^*$-convex decomposition of $\Phi$ with $\Phi_j\in\mathrm{CCP}(\mcl{A},\B{\mcl{H}})$ and $T_j\in\B{\mcl{H}}$, then $\norm{\Phi_j}=1$ for $j=1,2$.   
\end{lemma}

\begin{proof}
    Since $T_j^*T_j$ is positive invertible there exists a scalar $s_j\in (0,\infty)$ such that $T_j^*T_j\geq s_jI$. Now, since $\norm{\Phi_j(1)}=\norm{\Phi_j}\leq 1$, we get
    \begin{align*}
        \Phi(1) &={T_1}^*\Phi_1(1)T_1 + {T_2}^*\Phi_2(1)T_2 \\
                &\leq \norm{\Phi_1}T_1^*T_1+\norm{\Phi_2}T_2^*T_2 \\
                &= I - \big((1-\norm{\Phi_1})T_1^*T_1+(1-\norm{\Phi_2})T_2^*T_2\big)\\
                &\leq\big(1-(s_1(1-\norm{\Phi_1}) + s_2(1-\norm{\Phi_2}))\big)I.
    \end{align*}
    As $\norm{\Phi(1)}=\norm{\Phi}=1$, from the above, we get $s_1(1-\norm{\Phi_1}) + s_2(1-\norm{\Phi_2})=0$. Thus, $\norm{\Phi_1}=\norm{\Phi_2}=1$.
\end{proof}

\subsection{\texorpdfstring{$C^*$}{C*}-extreme points of $\mathrm{CCP}^\times$ maps}

To study $C^*$-extreme points $\Phi$ of CCP-maps, first we consider the special case when $\Phi(1)$ is invertible. 

\begin{lemma}
    The set
    \begin{align*}
        \mathrm{CCP}^\times(\mcl{A},\B{\mcl{H}})&:=\{\Phi\in\mathrm{CP}(\mcl{A},\B{\mcl{H}}):\Phi\mbox{ is contractive and }\Phi(1)\mbox{ is invertible}\}
    \end{align*}
    is a  $C^*$-convex set. 
\end{lemma}

\begin{proof}
  Let $\Phi_j \in \mathrm{CCP}^\times(\mcl{A},\B{\mcl{H}})$ and $T_j\in\B{\mcl{H}}$ with $\sum_{j=1}^nT_j^*T_j=1$. Then
  \begin{align*}
      0\leq\sum_{j=1}^nT_j^*\Phi_j(1)T_j\leq \sum_{j=1}^nT_j^*T_j=I
  \end{align*}
  so that $\Phi:= \sum_{j=1}^n\mathrm{Ad}_{T_j}\circ\Phi_j$ is a contraction. Let
  \begin{align*}
      T:=\Matrix{T_1\\T_2\\\vdots\\T_n}\in\B{\mcl{H},\oplus_{j=1}^n\mcl{H}},
      \qquad
      P:=\Matrix{\Phi_1(1)&&&\\&\Phi_2(1)&&\\&&\ddots&\\&&&\Phi_n(1)}\in\B{\oplus_{j=1}^n\mcl{H}}.
  \end{align*}
  Note that $\Phi(1) = \sum_{j=1}^n {T_j}^*\Phi_j(1)T_j = T^*PT$. 
  Since $P$ is positive and invertible, there exists a scalar $\alpha>0$ such that $\ip{z,Pz}\geq\alpha\norm{z}^2$ for all $z\in\oplus_{j=1}^n\mcl{H}$. Now, since $T$ is an isometry we have  
  \begin{align*}
      \ip{x,T^*PTx}= \ip{Tx,PTx}\geq \alpha \norm{Tx}^2= \alpha \norm{x}^2,\qquad\forall~x\in\mcl{H}.
  \end{align*}
  Hence $\Phi(1)=T^*PT$ is a positive invertible element so that $\Phi \in \mathrm{CCP}^\times(\mcl{A},\B{\mcl{H}})$. Therefore, $\mathrm{CCP}^\times(\mcl{A},\B{\mcl{H}})$ is a $C^*$-convex set. 
\end{proof}


Our main aim is to determine the structure of $C^*$-extreme points of CCP maps. Note that if $\Phi\in\mathrm{CCP}_{C^*-ext}(\mcl{A},\B{\mcl{H}})$ and $\Phi(1)$ is invertible, then by \eqref{eq-C-star-ext-subset}, $\Phi\in\mathrm{CCP}^\times_{C^*-ext}(\mcl{A},\B{\mcl{H}})$. This observation leads us to the following:

\begin{proposition}\label{prop-CCP-prime=UCP-ext}
    \begin{align*}
      \mathrm{CCP}^\times_{C^*-ext}(\mcl{A},\B{\mcl{H}}) = \mathrm{UCP}_{C^*-ext}(\mcl{A},\B{\mcl{H}}) 
    \end{align*} 
\end{proposition}

\begin{proof} 
    By Proposition \ref{prop-UCP-C-Cstar} we have $\mathrm{UCP}_{C^*-ext}(\mcl{A},\B{\mcl{H}}) \subseteq  \mathrm{CCP}^\times_{C^*-ext}(\mcl{A},\B{\mcl{H}}).$
    To prove the reverse inequality, let $\Phi \in \mathrm{CCP}^\times_{C^*-ext}(\mcl{A},\B{\mcl{H}})$ and $\Phi(1)=P$. Then, $\Phi = \mathrm{Ad}_{T_1}\circ\Phi_1 + \mathrm{Ad}_{T_2}\circ\Phi_2$ is a proper $C^*$-convex decomposition of $\Phi$, where $T_1 = \frac{1}{\sqrt{2}}P^{\frac{1}{2}}$, $T_2 = \frac{1}{\sqrt{2}}(2I-P)^{\frac{1}{2}}$ and $\Phi_j= \mathrm{Ad}_{\frac{1}{\sqrt{2}}{T_j}^{-1}}\circ\Phi \in \mathrm{CCP}^\times(\mcl{A},\B{\mcl{H}})$. Hence there exists a unitary $U \in \B{\mcl{H}}$ such that $\Phi= \mathrm{Ad}_U\circ\Phi_1$. In particular, $P=\Phi(1)=U^*\Phi_1(1)U=I$ so that $\Phi$ is unital. Hence, again by Proposition \ref{prop-UCP-C-Cstar}, we have $\Phi \in \mathrm{UCP}_{C^*-ext}(\mcl{A},\B{\mcl{H}})$.
\end{proof}

We conclude this subsection with a brief analysis of the linear extreme points of $\mathrm{CCP}^\times$ maps. Observe that 
    \begin{align*}
        \mathrm{CCP}^\times(\mcl{A},\B{\mcl{H}})&=\bigcup_{\substack{P\in\B{\mcl{H}}_{+}^{inv}\\\norm{P}\leq 1}}\mathrm{CP}^{(P)}(\mcl{A},\B{\mcl{H}}),
    \end{align*}
    where $\B{\mcl{H}}_{+}^{inv}$ denotes the set of all positive invertible elements in $\B{\mcl{H}}$.

 \begin{proposition}\label{prop-CCP-prime-ext-bounds} 
    \begin{align*}
      \mathrm{UCP}_{ext}(\mcl{A},\B{\mcl{H}}) \subseteq \mathrm{CCP}^\times_{ext}(\mcl{A},\B{\mcl{H}}) \subseteq \bigcup_{P\in\B{\mcl{H}}_+^{inv}, \norm{P}=1}\mathrm{CP}^{(P)}_{ext}(\mcl{A},\B{\mcl{H}})
    \end{align*} 
\end{proposition} 

\begin{proof}
   The first inclusion follows from Remark \ref{rmk-UCP-ext-C-ext}. Now, to see the second inclusion, let $\Phi\in \mathrm{CCP}^\times_{ext}(\mcl{A},\B{\mcl{H}})$. Then it follows from the definition of extreme points that $\Phi\in\mathrm{CP}^{(P)}_{ext}(\mcl{A},\B{\mcl{H}})$, where $P:=\Phi(1)\in\B{\mcl{H}}_+^{inv}$ is a contraction so that $0 < \norm{P}\leq 1$. If possible assume that $0<\norm{P}<1$. Then there exists $s\neq t \in (0,1)\backslash\{\norm{P}\}$ such that $\norm{P}= \frac{1}{2}s + \frac{1}{2}t$. Thus, $\Phi=\frac{1}{2}\big(\frac{s}{\norm{P}}\Phi\big) + \frac{1}{2}\big(\frac{t}{\norm{P}}\Phi\big)$ is a proper convex combination of $\mathrm{CCP}^\times$-maps, which leads to a contradiction since $\Phi$ is an extreme point. Therefore $\norm{P}=1$. 
\end{proof}

\begin{example}\label{eg-CCP-prime-CP-proper} 
 The inclusions in the above proposition can be possibly strict. To see this let $I\neq P\in\B{\mcl{H}}_+^{inv}$ be such that $\norm{P}=1$ and  $\dim(\mcl{H})>1$. 
 \begin{enumerate}[label=(\roman*)]
    \item  If $\Phi\in\mathrm{CCP}^\times(\mcl{A},\B{\mcl{H}})$ is a pure CP-map with $\norm{\Phi}=1$, then $\Phi\in\mathrm{CCP}^\times_{ext}(\mcl{A},\B{\mcl{H}})$. For, let  $\Phi=t\Phi_1 + (1-t)\Phi_2$ be a proper convex combination of $\Phi_j \in \mathrm{CCP}^\times(\mcl{A,\B{H}})$. Then $t\Phi_1 \leq_{cp} \Phi$ so that $t\Phi_1=s\Phi$ for some $s\in [0,1]$. Therefore $t\norm{\Phi_1(1)}=s\norm{\Phi}$, and using Lemma \ref{lem-norm-Phi-j} we conclude that $s=t$, i.e., $\Phi_1=\Phi$.  Similarly, we can show $\Phi_2=\Phi$. Hence, $\Phi \in \mathrm{CCP}^\times_{ext}(\mcl{A,\B{H}})$. In particular, $\Phi:=\mathrm{Ad}_{P^{\frac{1}{2}}}$ is a pure CP-map on $\B{\mcl{H}}$ with $\norm{\Phi}=\norm{P}=1$, and hence linear extreme point of $\mathrm{CCP}^\times(\B{\mcl{H}},\B{\mcl{H}})$, but $\Phi$ is not unital. This shows that the first inclusion can be strict.  
    \item Let $\psi: \mcl{A}\to \mbb{C}$ be a pure state. Consider the CP-map $\Phi:\mcl{A} \to \B{\mcl{H}}$ defined by $\Phi(\cdot) := \psi(\cdot)P$. By Proposition \ref{prop-PC-ext-non-empty} we have $\Phi\in \mathrm{CP}^{(P)}_{ext}(\mcl{A,\B{H}})$. Now let $\Phi_1(\cdot):= \psi(\cdot)P^2$ and $\Phi_2(\cdot):= \psi(\cdot)(2P-P^2)$. Note that $\Phi_1,\Phi_2\in\mathrm{CCP}^\times(\mcl{A},\B{\mcl{H}})$ as $P^2$ and $2P-P^2$ are positive invertible contractions. Clearly $\Phi = \frac{1}{2}\Phi_1 + \frac{1}{2}\Phi_2$ so that $\Phi \notin \mathrm{CCP}^\times_{ext}(\mcl{A},\B{\mcl{H}})$. Thus, the second inclusion in the above proposition is not equality in general. 
 \end{enumerate}
\end{example}

\begin{example}\label{eg-CC-prime-UCP-ext}
    If $\Phi\in\mathrm{CCP}^\times_{ext}(\mcl{A},\B{\mcl{H}})$, then one can easily see that $\widehat{\Phi}\in\mathrm{UCP}_{ext}(\mcl{A},\B{\mcl{H}})$. But the converse is not true in general. For example, let $P=\frac{1}{2}I \in \B{\mcl{H}}$ and fix $\Psi\in\mathrm{UCP}_{ext}(\mcl{A},\B{\mcl{H}})\neq\emptyset$. We observe that $\Phi:=\frac{1}{2}\Psi \in \mathrm{CP}^{(P)}_{ext}(\mcl{A},\B{\mcl{H}})$ and $\widehat{\Phi} = \Psi \in \mathrm{UCP}_{ext}(\mcl{A},\B{\mcl{H}})$. Since $\norm{P}<1$, by the above proposition, $\Phi \notin \mathrm{CCP}^\times_{ext}(\mcl{A},\B{\mcl{H}})$.  
\end{example}

\begin{corollary} 
 If $\dim(\mcl{H}) < \infty$, then,
    \begin{align*}
        \mathrm{CCP}^\times_{C^*-ext}(\mcl{A},\B{\mcl{H}}) \subseteq \mathrm{CCP}^\times_{ext}(\mcl{A},\B{\mcl{H}})
    \end{align*}
\end{corollary}

\begin{proof}
   Follows from Propositions \ref{prop-CCP-prime=UCP-ext}, \ref{prop-FaMo} and \ref{prop-CCP-prime-ext-bounds}.
\end{proof}


\subsection{\texorpdfstring{$C^*$}{C*}-extreme points of $\mathrm{CCP}$ maps}

 Recall that, from Proposition \ref{prop-UCP-C-Cstar}, we have 
 $$\mathrm{UCP}_{C^*-ext}(\mcl{A},\B{\mcl{H}})\subseteq\mathrm{CCP}_{C^*-ext}(\mcl{A},\B{\mcl{H}}).$$
  We know that $\mathrm{UCP}_{C^*-ext}(\mcl{A},\B{\mcl{H}})$ is non-empty $C^*$-convex set (see \cite{FaMo97}), and hence $\mathrm{CCP}_{C^*-ext}(\mcl{A},\B{\mcl{H}})$ is also non-empty. Note that if $0\neq \Phi\in\mathrm{CCP}_{C^*-ext}(\mcl{A,\B{\mcl{H}}})$, then $\norm{\Phi}=1$. For, if $0 < \norm{\Phi}< 1$, then as in Proposition \ref{prop-CCP-prime-ext-bounds}, we choose $s\neq t \in (0,1)\backslash\{\norm{\Phi}\}$ such that $\norm{\Phi}= \frac{1}{2}s + \frac{1}{2}t$. Thus $\Phi=\frac{1}{2}\big(\frac{s}{\norm{\Phi}}\Phi\big) + \frac{1}{2}\big(\frac{t}{\norm{\Phi}}\Phi\big)$ is a proper $C^*$-convex combination of $\mathrm{CCP}$-maps. But, $\Phi$ is not unitarily equivalent to $\frac{s}{\norm{\Phi}}\Phi$ as their norms are different. This is a contradiction. Hence $\norm{\Phi}=1$.

\begin{lemma}\label{lem-CCP-proj}
    Let $\Phi \in \mathrm{CCP}_{C^*-ext}(\mcl{A},\B{\mcl{H}})$ and $\ran{\Phi(1)}$ is closed. Then $\Phi(1)$ is a projection.
\end{lemma}

\begin{proof}
    Let $P:=\Phi(1)$ so that $\Phi\in\mathrm{CP}^{(P)}(\mcl{A},\B{\mcl{H}})$. If $P=0$, then nothing to prove. So assume $P\neq 0$. Now, if $\ker{P}=\{0\}$, then $P$ is invertible so that  
    \begin{align*}
       \Phi \in \mathrm{CCP}_{C^*-ext}(\mcl{A},\B{\mcl{H}}) 
\bigcap\mathrm{CCP}^\times(\mcl{A},\B{\mcl{H}})\subseteq\mathrm{CCP}^\times_{C^*-ext}(\mcl{A},\B{\mcl{H}}).
    \end{align*}
    Hence, by Prop \ref{prop-CCP-prime=UCP-ext},  $\Phi(1)=I$.
     Now, if $\ker{P}\neq\{0\}$, then let $\mcl{H}_0,P_0,\Phi_0$ be as in Lemma \ref{lem-P-P0-Phi-Phi0}. Set 
 \begin{align*}
     T_1=\frac{1}{\sqrt{2}}\Matrix{P_0^{\frac{1}{2}}&0\\0&I}
     \qquad\mbox{and}\qquad
     T_2=\frac{1}{\sqrt{2}}\Matrix{(2I-P_0)^{\frac{1}{2}}&0\\0&I}
 \end{align*}   
    in $\B{\mcl{H}}$ and let $\Phi_j:= \mathrm{Ad}_{\frac{1}{\sqrt{2}}{T_j}^{-1}}\circ\Phi \in \mathrm{CCP}(\mcl{A},\B{\mcl{H}})$ for $j=1,2$. Then $\Phi = \sum_{j=1}^2\mathrm{Ad}_{T_j}\circ\Phi_j$ is a proper $C^*$-convex decomposition, and hence there exists a unitary $U \in \B{\mcl{H}}$ such that $\Phi= \mathrm{Ad}_U\circ\Phi_1$. Since $\Phi_1(1)=\sMatrix{I&0\\0&0}$ is a projection, it follows that $\Phi(1)$ is also a projection.
\end{proof}

 In the above lemma, we are uncertain whether the assumption that $\ran{\Phi(1)}$ is closed follows automatically if $\Phi$ is a $C^*$-extreme point.

\begin{lemma}\label{lem-proj-unitary}
    Let $S\in\B{\mcl{H}}$ be invertible, $P\in\B{\mcl{H}}$ be a projection and $\Phi,\Psi \in \mathrm{CP}^{(P)}(\mcl{A,\B{\mcl{H}}})$ be such that  $\Psi = \mathrm{Ad}_S\circ\Phi$. Then there exists a unitary $U\in\B{\mcl{H}}$ such that $\Psi = \mathrm{Ad}_U\circ\Phi$.   
\end{lemma}

\begin{proof}
   Assume that $P\neq 0$. If $P=I$, then $S$ will be an invertible isometry, and therefore a unitary. So assume that $P\neq I$. Then $\ker{P}\neq\{0\}$, and let $\mcl{H}_0=\clran{P}$. With respect to the decomposition $\mcl{H}=\mcl{H}_0\oplus\mcl{H}_0^\perp$, as in Lemma \ref{lem-P-P0-Phi-Phi0}, we write 
   \begin{align*}
       P= \Matrix{I&0\\0&0},\quad 
       \Phi = \Matrix{\Phi_0&0\\0&0},\quad
       \Psi = \Matrix{\Psi_0&0\\0&0},
   \end{align*}
   where $\Phi_0,\Psi_0 \in \mathrm{UCP}(\mcl{A},\B{\mcl{H}_0})$. Now, $\Psi = \mathrm{Ad}_S\circ\Phi$ implies that $P=S^*PS$. Then, by Lemma \ref{lem-P-S-decomp}, $S$ has the block matrix form $S=\sMatrix{S_1&0\\S_2&S_3}$ with $S_1\in\B{\mcl{H}_0}$ invertible. Also, from $P=S^*PS$, it follows that $S_1$ is an isometry and $\Psi_0 = \mathrm{Ad}_{S_{1}}\circ\Phi_0$. Then $U:=\sMatrix{S_{1}&0\\0&I}$ is a unitary such that $\Psi = \mathrm{Ad}_U\circ\Phi$.
\end{proof}

Now we are ready to prove the main theorem of this section. 

\begin{theorem}\label{thm-CCP-C-star-ext} 
    If $\dim(\mcl{H}) < \infty$, then
    \begin{align*}
       \mathrm{CCP}_{C^*-ext}(\mcl{A,\B{\mcl{H}}})= \bigcup_{P=P^2=P^*}\mathrm{CP}^{(P)}_{C^*-ext}(\mcl{A,\B{\mcl{H}}}).
    \end{align*}
\end{theorem}

\begin{proof}
Let $P \in \B{\mcl{H}}$ be a projection and $\Phi \in \mathrm{CP}^{(P)}_{C^*-ext}(\mcl{A,\B{\mcl{H}}})$. Suppose $\Phi=\mathrm{Ad}_{T_1}\circ\Phi_1+\mathrm{Ad}_{T_2}\circ\Phi_2$ be a proper $C^*$-convex decomposition of $\Phi$ with $\Phi_j \in \mathrm{CCP}(\mcl{A,\B{\mcl{H}}})$ and $T_j\in\B{\mcl{H}}$. Then, 
    \begin{align*}
        P=\Phi(1)= {T_1}^*\Phi_1(1)T_1 + {T_2}^*\Phi_2(1)T_2,
    \end{align*}
    where $0\leq\Phi_j(1)\leq I$, and hence by \cite[Proposition 26]{LoPa81} and \cite{Wei02}, there exists a unitary $U_j\in\B{\mcl{H}}$ such that $\Phi_j(1)={U_j}^*PU_j$ for $j=1,2$. Thus 
    \begin{align*}
        P=(U_1T_1)^*PU_1T_1 +(U_2T_2)^*PU_2T_2.
    \end{align*}
    Let $\wtilde{\Phi}_j:= \mathrm{Ad}_{U_j^*}\circ\Phi_j\in \mathrm{CP}^{(P)}(\mcl{A,\B{\mcl{H}}})$ for $j=1,2$. Then
    \begin{align*}
        \Phi= \mathrm{Ad}_{T_1}\circ\Phi_1+\mathrm{Ad}_{T_2}\circ\Phi_2 
            = \mathrm{Ad}_{U_1T_1}\circ\wtilde{\Phi}_1+ \mathrm{Ad}_{U_2T_2}\circ\wtilde{\Phi}_2
    \end{align*}
    is a proper $P$-$C^*$-convex combination of $\wtilde{\Phi}_j$'s. Since $\Phi \in \mathrm{CP}^{(P)}_{C^*-ext}(\mcl{A,\B{\mcl{H}}})$ there exists invertible $S_j\in \B{\mcl{H}}$ such that $\wtilde{\Phi}_j = \mathrm{Ad}_{S_j}\circ\Phi$ for $j=1,2$. Now, by Lemma \ref{lem-proj-unitary}, we can choose $S_j$ to be unitary. Therefore, $\Phi_j = \mathrm{Ad}_{V_j}\circ\Phi$, where $V_j:=S_jU_j$ is unitary for $j=1,2$. Therefore, $\Phi \in \mathrm{CCP}_{C^*-ext}(\mcl{A,\B{\mcl{H}}})$. 
    
    Conversely, let $\Phi\in \mathrm{CCP}_{C^*-ext}(\mcl{A},\B{\mcl{H}})$. Then, by Lemma \ref{lem-CCP-proj}, $P:=\Phi(1)$ is a projection. Assume that $P\neq 0$. If $P=I$, then from Proposition \ref{prop-UCP-C-Cstar}  we have $\Phi\in \mathrm{UCP}_{C^*-ext}(\mcl{A},\B{\mcl{H}})=\mathrm{CP^{(I)}}_{C^*-ext}(\mcl{A},\B{\mcl{H}})$. So assume $P\neq I$. Let $\mcl{H}_0, P_0,\Phi_0$ be as in Lemma \ref{lem-P-P0-Phi-Phi0}. Since $P$ is a projection, we have $P_0 = I\in\B{\mcl{H}_0}$ and $\Phi_0 \in \mathrm{UCP}(\mcl{A},\B{\mcl{H}_0})$. Now, we show that  $\Phi_0\in\mathrm{UCP}_{C^*-ext}(\mcl{A},\B{\mcl{H}_0})$ so that, by Proposition \ref{prop-Phi-Phi0}, $\Phi\in\mathrm{CP^{(P)}}_{C^*-ext}(\mcl{A},\B{\mcl{H}})$. So let $\Phi_0=\sum_{j=1}^n\mathrm{Ad}_{T_j}\circ\Psi_j$ be a proper $C^*$-convex decomposition of $\Phi_0$, where $\Psi_j\in\mathrm{UCP}(\mcl{A},\B{\mcl{H}_0})$ and $T_j\in\B{\mcl{H}_0}$ invertible such that $\sum_{j=1}^nT_j^*T_j=I$. Then $\widetilde{\Psi}_j:=\sMatrix{\Psi_j&0\\0&0}\in\mathrm{CCP}(\mcl{A},\B{\mcl{H}})$ and $\widetilde{T}_j:=\sMatrix{T_j&0\\0&\frac{1}{\sqrt{n}}I}\in\B{\mcl{H}}$ invertible are such that $\Phi=\sum_{j=1}^n\mathrm{Ad}_{\widetilde{T}_j}\circ\widetilde{\Psi}_j$ is a proper $C^*$-convex decomposition of $\Phi$. Hence there exists unitary $U_j=\sMatrix{X_j&Y_j\\Z_j&W_j}\in\B{\mcl{H}_0\oplus\mcl{H}_0^\perp}$ such that 
    \begin{align*}
        \Matrix{\Phi_0(\cdot)&0\\0&0}
                =\Phi(\cdot)=\mathrm{Ad}_{U_j}\circ\widetilde{\Psi}_j(\cdot)
                =\Matrix{X_j^*\Psi_j(\cdot)X_j&X_j^*\Psi_j(\cdot)Y_j\\ Y_j^*\Psi_j(\cdot)X_j&Y_j^*\Psi_j(\cdot)Y_j}
                ,\qquad\forall~1\leq j\leq n.
    \end{align*}
    Since $\Phi_0,\Psi_j$ are unital, from the above equation, we get $X_j^*X_j=I$ and  $Y_j^*Y_j=0$, i.e., $Y_j=0$ for all $1\leq j\leq n$. Then $U_j$ unitary implies $X_j$ is unitary, and $\Phi_0=\mathrm{Ad}_{X_j}\circ\Psi_j$ for $1\leq j\leq n$. Hence $\Phi_0\in{\mathrm{UCP}_{C^*-ext}(\mcl{A},\B{\mcl{H}_0})}$. This completes the proof.
\end{proof}

\begin{remark}\label{rmk-CCP-Cstar-structure}
    Suppose $\dim(\mcl{H})<\infty$. Then, from Proposition \ref{prop-Phi-Phi0} and Theorem \ref{thm-CCP-C-star-ext}, we have the following: A contractive CP map $\Phi$ is a $C^*$-extreme point of  $\mathrm{CCP}(\mcl{A,\B{\mcl{H}}})$ if and only if there exists a closed subspace $\mcl{H}_0\subseteq\mcl{H}$ and $\Psi\in\mathrm{UCP}_{C^*-ext}(\mcl{A},\B{\mcl{H}_0})$ such that
    \begin{align}\label{eq-CCP-Cstar-structure}
     \Phi=\Matrix{\Psi&0\\0&0},
    \end{align} 
    with respect to the decomposition $\mcl{H}=\mcl{H}_0\oplus\mcl{H}_0^\perp$. Also, by Theorem \ref{thm-CCP-C-star-ext}  and Lemma \ref{lem-proj-unitary}, $\Phi$ must be of the form \eqref{eq-CP-P-Cstar-structure} with $S$ unitary, where $P:=\Phi(1)$ is a projection. In particular, if $\mcl{A}$ is a commutative unital $C^*$-algebra, then by Proposition \ref{prop-FaMo} we conclude that $\Phi\in\mathrm{CCP}_{C^*-ext}(\mcl{A},\B{\mcl{H}})$ if and only if $\Phi$ is a $\ast$-homomorphism. 
\end{remark}

\begin{note}
 We observe that by using Lemma \ref{lem-CCP-proj}, Corollary \ref{cor-invertible-conjugate-positive-inequality} and \cite[Proposition 26]{LoPa81}, one can derive the structure \ref{eq-CCP-Cstar-structure} without invoking $P$-$C^*$-extreme points. However, for infinite-dimensional Hilbert spaces, we are unsure whether this can be accomplished. The proof of the above theorem demonstrates that the inclusion
  \begin{align}\label{eq-C-star-ext-general-inclusion}
        \bigcup_{P=P^2=P^*}\mathrm{CP}^{(P)}_{C^*-ext}(\mcl{A,\B{\mcl{H}}})\subseteq\mathrm{CCP}_{C^*-ext}(\mcl{A,\B{\mcl{H}}})
    \end{align}
 holds for any Hilbert space $\mcl{H}$. 
 If we can prove that $\Phi(1)$ is a projection for any $\Phi\in\mathrm{CCP}_{C^*-ext}(\mcl{A,\B{\mcl{H}}})$, then it will follow that each $\Phi\in \mathrm{CCP}_{C^*-ext}(\mcl{A,\B{\mcl{H}}})$ must be of the form $\Phi=\sMatrix{\Psi&0\\0&0}$ for some $\Psi\in \mathrm{UCP}_{C^*-ext}(\mcl{A},\B{\mcl{H}_0})$ and closed subspace $\mcl{H}_0\subseteq\mcl{H}$. 
 Additionally, if Proposition \ref{prop-Phi-Phi0} (ii) holds for any $\mcl{H}$ and $P=\Phi(1)$ projection, then one can  arrive at the structure \eqref{eq-CCP-Cstar-structure}. 
\end{note}
 

Next we prove a Krein-Milman type theorem for the $C^*$-convex set of CCP-maps.

\begin{lemma}\label{lem-conhull-proj}
    Let $0\neq P\in\B{\mcl{H}}$ be a projection. Then  
    \begin{align*}
    P\mbox{-}C^*{\mbox{-}}con(\mcl{S}) \subseteq C^*{\mbox{-}}con(\mcl{S}),
    \end{align*}
   for any subset $\mcl{S}\subseteq\mathrm{CP}^{(P)}(\mcl{A},\B{\mcl{H}})$.
\end{lemma}

\begin{proof}
    If $P=I$, nothing to prove. So assume $P\neq I$ and let $\Phi\in P\mbox{-}C^*{\mbox{-}}con(\mcl{S})$. Assume $\Phi = \sum_{j=1}^n \mathrm{Ad}_{T_j}\circ\Phi_j$, where $\Phi_j \in \mcl{S}$ and $T_j \in \B{\mcl{H}}$ is such that $\sum_{j=1}^n {T_j}^*PT_j = P$. Let $\mcl{H}_0=\ran{P}$. With respect to the decomposition $\mcl{H}=\mcl{H}_0\oplus\mcl{H}_0^\perp$ write 
    \begin{align*}
        P=\Matrix{I&0\\0&0},\quad
        T_j = \Matrix{X_j&Y_j\\Z_j&W_j},\quad
        \Phi_j=\Matrix{\Psi_j&0\\0&0},\quad
        \Phi=\Matrix{\Phi_0&0\\0&0},
    \end{align*}
    where $\Phi_0,\Psi_j \in \mathrm{UCP}(\mcl{A},\B{\mcl{H}_0})$ are as in \eqref{eq-P-P0-Phi-Phi0}. Set $S_j = \sMatrix{X_j&0\\0&\frac{1}{\sqrt{n}}I}$ for all $1\leq j\leq n$. Then $\Phi=\sum_{j=1}^n \mathrm{Ad}_{S_j}\circ\Phi_j \in C^*{\mbox{-}}con(\mcl{S})$. 
\end{proof}

\begin{theorem}\label{thm-KMT-CCP}
    Suppose $\dim(\mcl{H}) < \infty$. Then,
    \begin{align*}
       \mathrm{CCP}(\mcl{A},\B{\mcl{H}})=
       \ol{C^*{\mbox{-}}con}\Big(\mathrm{CCP}_{C^*-ext}(\mcl{A},\B{\mcl{H}})\Big),
   \end{align*}
   where the closure is taken with respect to the BW-topology on $\mathrm{CP}(\mcl{A},\B{\mcl{H}})$. 
\end{theorem}

\begin{proof}
    First, we observe that if $P\in\B{\mcl{H}}$ is a non-zero projection, then by Corollary \ref{cor-KM-CP-P-fd} and Lemma \ref{lem-conhull-proj}, we get 
    \begin{align}\label{eq-KMT-inclu}
      \mathrm{CP}^{(P)}(\mcl{A},\B{\mcl{H}}) 
      &= \ol{P\mbox{-}C^*{\mbox{-}}con}\Big(\mathrm{CP}^{(P)}_{C^*-ext}(\mcl{A},\B{\mcl{H}})\Big) \notag\\
      &\subseteq \ol{C^*{\mbox{-}}con}\Big(\mathrm{CP}^{(P)}_{C^*-ext}(\mcl{A},\B{\mcl{H}})\Big) \notag\\
      &\subseteq \ol{C^*{\mbox{-}}con}\Big(\mathrm{CCP}_{C^*-ext}(\mcl{A},\B{\mcl{H}})\Big).
     \end{align}
    The above inclusions hold trivially if $P=0$. Now, let $0\neq \Phi \in \mathrm{CCP}(\mcl{A},\B{\mcl{H}})$ and $P=\Phi(1)$. If $\ker{P}\neq\{0\}$, then take $\mcl{H}_0,\Phi_0,P_0$ as in Lemma \ref{lem-P-P0-Phi-Phi0} and set 
    \begin{align*}
        T_1 = \Matrix{{P_0}^{\frac{1}{2}}&0\\0&\frac{1}{\sqrt{2}}I},\quad
        T_2=\Matrix{(I-P_0)^{\frac{1}{2}}&0\\0&\frac{1}{\sqrt{2}}I},\quad
        \Phi_1 = \Matrix{ \widehat{\Phi}_0 &0\\0&0},\quad
        \Phi_2 = \Matrix{0&0\\0&0}.
    \end{align*}
    Since $\Phi_j(1)$'s are projections, from \eqref{eq-KMT-inclu}, we conclude that 
    \begin{align*} 
        \Phi = \sum_{j=1}^2\mathrm{Ad}_{T_j}\circ\Phi_j\in
        \ol{C^*{\mbox{-}}con}\Big(\mathrm{CCP}_{C^*-ext}(\mcl{A},\B{\mcl{H}})\Big).
    \end{align*}
    Similarly, if $\ker{P}=\{0\}$, then $P$ is invertible so that, again from \eqref{eq-KMT-inclu}, we have
    \begin{align*}
       \Phi=\mathrm{Ad}_{P^{\frac{1}{2}}}\circ\widehat{\Phi}+\mathrm{Ad}_{(I-P)^{\frac{1}{2}}}\circ 0 \in
        \ol{C^*{\mbox{-}}con}\Big(\mathrm{CCP}_{C^*-ext}(\mcl{A},\B{\mcl{H}})\Big).
    \end{align*}
    This completes the proof.
\end{proof}

In the remainder of this section, we examine the linear extreme points of CCP maps and their relationship with $C^*$-extreme points.

\begin{proposition}\label{prop-CCP-ext-bounds}
    \begin{align}\label{eq-lin-extremes-bounds}
        \bigcup_{P=P^2=P^*}\mathrm{CP}^{(P)}_{ext}(\mcl{A},\B{\mcl{H}})  \subseteq\mathrm{CCP}_{ext}(\mcl{A},\B{\mcl{H}}) \subseteq \bigcup_{0\leq P \leq I, \norm{P}\in\{0,1\}}\mathrm{CP}^{(P)}_{ext}(\mcl{A},\B{\mcl{H}}).
    \end{align}
\end{proposition} 

\begin{proof}
    Let $P \in\B{\mcl{H}}$ be a projection and $\Phi \in \mathrm{CP}^{(P)}_{ext}(\mcl{A},\B{\mcl{H}})$. Suppose $\Phi = t\Phi_1 + (1-t)\Phi_2$ where $\Phi_1,\Phi_2 \in \mathrm{CCP}(\mcl{A},\B{\mcl{H}})$ and $t\in (0,1)$. Then, $P=t\Phi_1(1) + (1-t)\Phi_2(1)$ with $\Phi_1(1),\Phi_2(1)$ are positive contractions. Then by \cite[Proposition 54.2]{Con00} we must have $P=\Phi_1(1)=\Phi_2(1)$ so that $\Phi_1,\Phi_2 \in \mathrm{CP}^{(P)}(\mcl{A},\B{\mcl{H}})$. Furthermore, since $\Phi \in \mathrm{CP}^{(P)}_{ext}(\mcl{A},\B{\mcl{H}})$, we have $\Phi=\Phi_1=\Phi_2$. Hence, $\Phi \in \mathrm{CCP}_{ext}(\mcl{A},\B{\mcl{H}})$. This proves the first inclusion. 

    Now, to prove the second inclusion, let $\Phi\in \mathrm{CCP}_{ext}(\mcl{A},\B{\mcl{H}})$. Then, from the definition of linear extreme points, we observe that $\Phi\in\mathrm{CP}^{(P)}_{ext}(\mcl{A},\B{\mcl{H}})$, where $P=\Phi(1)$ is a positive contraction. If possible assume that $\norm{\Phi}=\norm{P}\in (0,1)$. Choose $s\neq t \in(0,1)\backslash\{\norm{P}\}$ such that $\norm{P}= \frac{1}{2}s + \frac{1}{2}t$. Then, $\Phi=\frac{1}{2}(\frac{s}{\norm{P}}\Phi) + \frac{1}{2}(\frac{t}{\norm{P}}\Phi)$ is a proper convex combination of $\mathrm{CCP}$ maps, which is not possible as $\Phi\in\mathrm{CCP}_{ext}(\mcl{A},\B{\mcl{H}})$. Hence $\norm{P}\in \{0,1\}$. This proves the second inclusion. 
\end{proof}

\begin{remark}
    In the above proposition, if we replace $\B{\mcl{H}}$ with an arbitrary commutative unital $C^*$-algebra $\mcl{B}$, then the first inclusion becomes an equality, i.e.,
    \begin{align*}
        \mathrm{CCP}_{ext}(\mcl{A,B})=\bigcup_{P=P^2=P^*}\mathrm{CP}^{(P)}_{ext}(\mcl{A,B}).
    \end{align*}
    We see this as follows. Assume that $\mcl{B}$ is commutative, and let $\Phi\in\mathrm{CCP}_{ext}(\mcl{A,B})$. Then, from the above proposition, $\Phi\in\mathrm{CP}_{ext}^{(P)}(\mcl{A,B})$ where $P=\Phi(1)\in\mcl{B}$ is a positive contraction of norm either zero or one. Note that $\Phi = \frac{1}{2}\Phi_1 + \frac{1}{2}\Phi_2$ where $\Phi_1= \mathrm{Ad}_{P^{\frac{1}{2}}}\circ\Phi , \Phi_2= \mathrm{Ad}_{(2-P)^{\frac{1}{2}}}\circ\Phi \in \mathrm{CCP}(\mcl{A,B})$.  Now, $\Phi \in\mathrm{CCP}_{ext}(\mcl{A,B})$ implies $\Phi = \Phi_1 = \Phi_2$ and therefore, $P$ is a projection. Thus, $\mathrm{CCP}_{ext}(\mcl{A,B})=\bigcup_{P=P^2=P^*}\mathrm{CP}^{(P)}_{ext}(\mcl{A,B})$. 
\end{remark}

\begin{example}
    The inclusions in Proposition \ref{prop-CCP-ext-bounds} can be possibly strict. To see this let $P\in\B{\mcl{H}}_+$ be such that $\norm{P}=1$, $P$ is not a projection and $\dim(\mcl{H})>1$. 
    \begin{enumerate}[label=(\roman*)]
        \item If $\Phi\in\mathrm{CCP}(\mcl{A},\B{\mcl{H}})$ is a pure CP-map with $\norm{\Phi}=1$, then as in Example \ref{eg-CCP-prime-CP-proper} (i), we get $\Phi\in\mathrm{CCP}_{ext}(\mcl{A},\B{\mcl{H}})$.  In particular, $\Phi:=\mathrm{Ad}_{P^{\frac{1}{2}}}$ is a pure CP-map on $\B{\mcl{H}}$ with $\norm{\Phi}=\norm{P}=1$,  and hence $\Phi\in\mathrm{CCP}_{ext}(\B{\mcl{H}}, \B{\mcl{H}})$. However, since $P$ is not a projection 
        $$\Phi\notin\bigcup_{P=P^*=P^2} \mathrm{CP}_{ext}^{(P)}(\B{\mcl{H}}, \B{\mcl{H}}).$$
        Thus in \eqref{eq-lin-extremes-bounds} the first inclusion can be strict. 
        \item Let $\psi: \mcl{A}\to \mbb{C}$ be a pure state and consider the CP-map $\Phi:\mcl{A} \to \B{\mcl{H}}$ defined by $\Phi(\cdot) := \psi(\cdot)P$. Then, by Proposition \ref{prop-PC-ext-non-empty}, we have $\Phi\in \mathrm{CP}^{(P)}_{ext}(\mcl{A,\B{H}})$. But, as in Example \ref{eg-CCP-prime-CP-proper} (ii), we can see that $\Phi\notin\mathrm{CCP}_{ext}(\mcl{A},\B{\mcl{H}})$. 
        This shows that in \eqref{eq-lin-extremes-bounds} the second inclusion can also be strict. 
    \end{enumerate}
\end{example}

\begin{remark}
 Let $P \in \B{\mcl{H}}$ be a positive contraction and $\psi: \mcl{A}\to \mbb{C}$ be a pure state.  We saw that  $\Phi(\cdot)=\psi(\cdot)P$ is an element of $\mathrm{CP}^{(P)}_{ext}(\mcl{A,\B{H}})$. But $\Phi \in \mathrm{CCP}_{ext}(\mcl{A,\B{H}})$ if and only if $P$ is a projection. For, if $P$ is a projection then by the above proposition $\Phi\in\mathrm{CP}^{(P)}_{ext}(\mcl{A},\B{\mcl{H}})\subseteq\mathrm{CCP}_{ext}(\mcl{A},\B{\mcl{H}})$. Now, if $P$ is not a projection, then as in Example \ref{eg-CCP-prime-CP-proper} (ii), we can see that $\Phi\notin\mathrm{CCP}_{ext}(\mcl{A},\B{\mcl{H}})$.
\end{remark}

\begin{corollary}\label{rmk-CCP-Cstar-linear-ext}
    If $\dim(\mcl{H}) < \infty$, then,
    \begin{align*}
        \mathrm{CCP}_{C^*-ext}(\mcl{A},\B{\mcl{H}}) \subseteq \mathrm{CCP}_{ext}(\mcl{A},\B{\mcl{H}})
    \end{align*}
\end{corollary}

\begin{proof}
    This follows from Theorem \ref{thm-CCP-C-star-ext}, along with Propositions \ref{prop-Cstar-linear-ext-finite} and \ref{prop-CCP-ext-bounds}.
\end{proof}

\section*{Acknowledgment}
The first author is supported by a CSIR fellowship (File No. 09/084(0780)/2020-EMR-I). The second author receives partial support from the IoE Project of MHRD (India) with reference no SB22231267MAETWO008573,  and is also partially funded by the NBHM grant (No. 02011/10/2023 NBHM (R.P) R\&D II/4225).

\bibliographystyle{alpha}

\end{document}